\documentclass[10pt]{article}
\usepackage{fancyhdr,a4wide}
\usepackage[mathscr]{eucal}
\usepackage{amsmath}
\usepackage{amssymb}
\usepackage{amsthm}
\usepackage{float}
\usepackage{caption}
\usepackage[caption = false]{subfig}
\usepackage{enumerate}
\usepackage{fullpage}
\usepackage{graphicx}
\usepackage{multicol}
\usepackage{tikz,  mathtools, soul, mathdots, yhmath}
\usetikzlibrary{calc,decorations.markings,arrows}

\usepackage{hyperref}

\input xy
\xyoption{all}

\usepackage[all]{xy}

\usepackage{color}

\title{Construction of Rank $2$ Indecomposable Modules in Grassmannian Cluster Categories} 
\author{Karin Baur, Dusko Bogdanic, and Jian-Rong Li}
\date{} 

\usepackage{pst-all}  

\definecolor{light-grey}{gray}{0.6}  

\begin{document}

\newtheorem{lm}{Lemma}[section]
\newtheorem{prop}[lm]{Proposition}
\newtheorem{satz}[lm]{Satz}

\newtheorem{corollary}[lm]{Corollary}
\newtheorem{cor}[lm]{Korollar}
\newtheorem{claim}[lm]{Claim}
\newtheorem{theorem}[lm]{Theorem}
\newtheorem*{thm}{Theorem}

\theoremstyle{definition}
\newtheorem{defn}[lm]{Definition}
\newtheorem{qu}{Question}
\newtheorem{ex}[lm]{Example}
\newtheorem{exas}[lm]{Examples}
\newtheorem{exc}[lm]{Exercise}
\newtheorem*{facts}{Facts}
\newtheorem{rem}[lm]{Remark}

\theoremstyle{remark}

\newcommand{\perm}{\operatorname{Perm}\nolimits}
\newcommand{\NN}{\operatorname{\mathbb{N}}\nolimits}
\newcommand{\ZZ}{\operatorname{\mathbb{Z}}\nolimits}
\newcommand{\QQ}{\operatorname{\mathbb{Q}}\nolimits}
\newcommand{\RR}{\operatorname{\mathbb{R}}\nolimits}
\newcommand{\CC}{\operatorname{\mathbb{C}}\nolimits}
\newcommand{\PP}{\operatorname{\mathbb{P}}\nolimits}
\newcommand{\cA}{\operatorname{\mathcal{A}}\nolimits}
\newcommand{\LL}{\operatorname{\Lambda}\nolimits}

\newcommand{\MM}{\operatorname{\mathbb{M}}\nolimits}
\newcommand{\Fk}{\mathcal{F}_{k,n}}
\newcommand{\Mk}{M_{k,n}}

\newcommand{\ad}{\operatorname{ad}\nolimits}
\newcommand{\im}{\operatorname{im}\nolimits}
\newcommand{\Char}{\operatorname{char}\nolimits}
\newcommand{\Aut}{\operatorname{Aut}\nolimits}

\newcommand{\id}{\operatorname{id}\nolimits}
\newcommand{\Id}{\operatorname{Id}\nolimits}
\newcommand{\ord}{\operatorname{ord}\nolimits}
\newcommand{\ggT}{\operatorname{ggT}\nolimits}
\newcommand{\lcm}{\operatorname{lcm}\nolimits}

\newcommand{\Gr}{\operatorname{Gr}\nolimits}

\maketitle

\begin{abstract} 
The category ${\rm CM}(B_{k,n}) $ of Cohen-Macaulay modules over a quotient $B_{k,n}$ of a 
preprojective algebra provides a categorification of the cluster algebra structure on the 
coordinate ring of the Grassmannian variety of $k$-dimensional subspaces in $\mathbb C^n$, 
\cite{JKS16}. Among the indecomposable modules in this category are the rank $1$ modules 
which are in bijection with $k$-subsets of $\{1,2,\dots,n\}$, and their explicit construction has been given 
by Jensen, King and Su. These are the building blocks of the category as any module in ${\rm CM}(B_{k,n}) $ can be filtered by them. 
In this paper we give an explicit construction of rank 2 modules. With this, we give all 
indecomposable rank 2 modules in the cases when $k=3$ and $k=4$. In particular, we cover the tame cases and 
go beyond them. We also characterise the modules among them which are uniquely determined by 
their filtrations. 
For $k\ge 4$, we exhibit infinite families of non-isomorphic rank 2 modules having the same filtration. 

\end{abstract}

\noindent
%
%
\section{Introduction} 
%
One of the key initial examples of Fomin and Zelevinsky's theory of cluster algebras \cite[\S 12.2]{FZ} was the homogeneous coordinate ring $\CC[\Gr(2,n)]$ of the Grassmannian of $2$-dimensional subspaces of $\CC^n$. Scott proved in~\cite{Scott06} that this cluster structure can be generalized to the coordinate ring $\CC[\Gr(k,n)]$. This has sparked a lot of research activities in cluster theory, e.g.  \cite{SW, gls, HL10, gssv, MuS, BKM16, JKS16, mr,fraser, SSBW}. 

An aditive categorification of the cluster algebra structure on the homogeneous coordinate ring 
of the Grassmannian variety of $k$-dimensional subspaces in $\mathbb C^n$ has been given by Geiss, Leclerc, and Schroer \cite{rigid, GLS08} in terms of a subcategory of the category of finite dimensional modules over the preprojective algebra of type $A_{n-1}$. Jensen, King, and Su \cite{JKS16} gave a new additive categorification of this cluster structure using the maximal Cohen-Macaulay modules over the completion of an algebra $B_{k,n}$ which is a quotient of the preprojective algebra of type $A_{n-1}$.  In the category ${\rm CM}(B_{k,n}) $ of Cohen-Macaulay modules over $B_{k,n}$, which is called the Grassmannian cluster category, among the indecomposable modules are the rank $1$ modules which are known to be in bijection with $k$-subsets of $\{1,2,\dots,n\}$, and their explicit construction has been given in \cite{JKS16}. For a given $k$-subset $I$, the corresponding rank 1 module is denoted by $L_I$. Also, we refer to $k$-subsets as rims, because of the way we use them to visualize rank 1 modules (see Section 2). Rank 1 modules are the building blocks of the category as any module in ${\rm CM}(B_{k,n}) $ can be filtered by rank $1$ modules (the filtration is noted in the profile of a module, \cite[Corollary 6.7]{JKS16}). The number of rank 1 modules appearing in the filtration of a given module is called the rank of that module. 

The aim of this paper is to explicitly construct rank 2 indecomposable Cohen-Macaulay $B_{k,n}$ modules in the cases when $k=3$ and $k=4$. In particular, we construct all indecomposable rank 2 modules in the tame cases $(3,9)$ and $(4,8)$, and more generally, for an arbitrary $k$, we construct all indecomposable modules of rank 2 whose rank 1 filtration layers $L_I$ and $L_J$ satisfy the condition $|I\cap J|\geq k-4$. 

Once we have the construction, 
we investigate the question of uniqueness. Here, the central notions are that of $r$-interlacing (Definition~\ref{interlacing}) and of the poset of a given rank 2 module (Section~\ref{sec:prelims}). 
If $I$ and $J$ are $k$-subsets of $\{1,\ldots, n\}$, then
$I$ and $J$ are said to be {\em $r$-interlacing} if there exist subsets 
$\{i_1,i_3,\dots,i_{2r-1}\}\subset I\setminus J$ and $\{i_2,i_4,\dots, i_{2r}\}\subset J\setminus I$ 
such that $i_1<i_2<i_3<\dots <i_{2r}<i_1$ (cyclically) 
and if there exist no larger subsets of $I$ and of $J$ with this property. 
The filtration layers of a module $M$ give a poset structure. In rank 2, if $I$ and $J$ are $r$-interlacing, the sets 
$I$ and $J$ form a number $r_1\le r$ of boxes in the so-called lattice diagram of $M$ (see Section~\ref{sec:prelims} for details 
on how we picture $M$ with its filtration layers). 
The associated poset is 
$1^{r_1}\mid 2$; the poset consists of a tree with one vertex of degree $r_1$ and $r_1$ leaves, it has  
dimension $1$ at the leaves and dimension 2 at the central vertex. See Figure~\ref{fig:boxes-poset}. 

\begin{figure}[H] 
\begin{center}
{\includegraphics[width = 8cm]{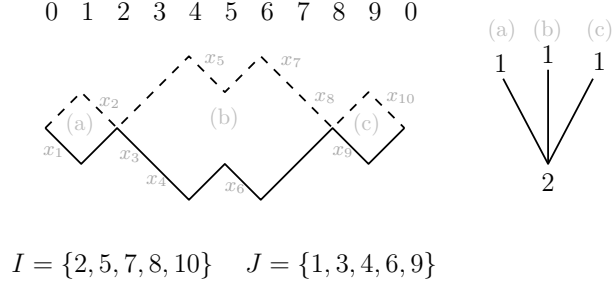}} 
\caption{The profile of  a  module  with $4$-interlacing layers forming 3 boxes  with poset $1^3\mid 2$. 
The dashed line shows the rim of $L_I$ with arrows $x_i$, $i\in I$ indicated. 
The solid line below is the rim of $L_J$, with arrows $x_i$, $i\in J$ indicated.} \label{fig:boxes-poset}
\end{center}
\end{figure}
A partial answer to the question of indecomposability of a rank 2 module in terms of its poset is given in the following proposition. 

\begin{prop}[\cite{BBGE}, Remark 3.2] \label{poset}
Let $M\in {\rm CM}(B_{k,n})$ be an indecomposable module with profile $I\mid J$. Then  $I$ and $J$ are  $r$-interlacing and their poset is $1^{r_1}\mid 2$, where $r\geq r_1\geq 3$.   
\end{prop}

This proposition tells us that when dealing with rank 2 indecomposable modules, we can assume that the poset of such a module  is of the form $1^{r_1}\mid 2$, for $r_1\geq 3$, and that its layers are $r$-interlacing, where $r\geq r_1\geq 3$. We say that $I$ and $J$ are almost tightly $3$-interlacing if $I\mid J$ has poset $1^3\mid 2$, and  $I\setminus J=\{a_1\}\cup\{a_2\}\cup \{a_3, \dots, a_{3+r}\}$, $J\setminus I=\{b_1\}\cup\{b_2\}\cup \{b_3, \dots, b_{3+r}\}$, $r\geq 0$, and $b_1<a_1<b_2<a_2<b_3< \dots <b_{3+r}<a_3< \dots < a_{3+r}.$ Our main results are the following two theorems. 

\begin{theorem} [Theorem~\ref{thm:unique}, Theorem~\ref{thm:infmod}]
An indecomposable rank $2$ module $M\in {\rm CM}(B_{k,n})$ is uniquely determined by its profile if and only if its poset is $1^3\mid 2$ and its layers are almost tightly $3$-interlacing. 
\end{theorem}

More precisely, in the case of $r$-interlacing rank 1 layers with poset $1^{r_1}\mid 2$, where $r\geq r_1\geq  4$, we show that there are infinitely many non-isomorphic rank 2 modules with the same profile, e.g.\ there are infinitely many non-isomorphic indecomposable rank 2 modules with filtration $\{1,3,5,7\}\mid \{2,4,6,8\}$ in the tame case $(4,8)$.  

\begin{theorem} [Theorem~\ref{thm:infmod}]
Let $M$ be an indecomposable rank $2$ module with profile $I\mid J$, where  $I$ and $J$ are $r$-interlacing with poset 
$1^{r_1}\mid 2$, where $r\geq r_1\geq  4$. Then there are infinitely many non-isomorphic rank $2$ indecomposable  modules  from 
${\rm CM}(B_{k,n})$ with profile $I \mid J$.
\end{theorem}

In the case $r=r_1=4$, we show that this infinite family of indecomposable modules with the profile $I\mid J$ is 
parameterized by the set $\mathbb C\setminus {\{0,1,-1\}}$ where two points from this set are identified if their sum is  $0$. For the filtration layers $I$ and $J$ of an indecomposable module with profile $I\mid J$, 
we construct all decomposable rank 2 modules that are extensions of these rank 1 modules, i.e.\ we construct all decomposable modules that appear as middle terms in short exact sequences with $I$ and $J$ as end terms.

The paper is organized as follows. In Section~\ref{sec:prelims}, we recall the definitions and key results
about Grassmannian cluster categories.  In Section~\ref{sec:tight-3}, we give the construction of rank 2 modules in the case when the layers are tightly $3$-interlacing. 
This covers in particular the tame case $(3,9)$ and almost all rank $2$ modules in the tame 
case $(4,8)$. 
Section~\ref{sec:non-tight3} is devoted to the cases of non-tightly $3$-interlacing layers. 
Section~\ref{sec:4-interlacing} is devoted to the case of tightly $4$-interlacing layers, which completes 
the case $(4,8)$. In the last section, we deal with the general case of $r$-interlacing, when $r\geq 4$, 
and we show that there are infinitely many non-isomorphic rank 2 indecomposable modules with the same filtration. 

\subsection*{Acknowledgments} 
We thank  Matthew Pressland and Alastair King for numerous helpful conversations. K. B. was supported by a Royal Society Wolfson Fellowship. She is currently on leave from the University of Graz. D.B.\ was supported by the Austrian Science Fund Project Number P29807-35. J.-R.L. was supported by the Austrian Science Fund (FWF): M 2633-N32 Lise Meitner Program.

\section{Preliminaries}\label{sec:prelims}

We follow the exposition from \cite{JKS16, BB, BBGE} in order to introduce notation and background results. Let $\Gamma_n$ be the quiver of the boundary algebra, with vertices $1,2,\dots, n$ 
on a cycle and arrows $x_i: i-1\to i$, $y_i:i\to i-1$. 
We write ${\rm CM}(B_{k,n})$ for the category of maximal Cohen-Macaulay modules for  
the completed path algebra $B_{k,n}$ of $\Gamma_n$, with relations 
$xy-yx$ and $x^k-y^{n-k}$ (at every vertex). The centre of $B_{k,n}$ is 
$Z:=\CC[|t|]$, where $t=\sum_ix_iy_i$.  For example,  when $n=5$ we have the quiver
\begin{center}
\begin{tikzpicture}[scale=1]
\newcommand{\radius}{1.5cm}
\foreach \j in {1,...,5}{
  \path (90-72*\j:\radius) node[black] (w\j) {$\bullet$};
  \path (162-72*\j:\radius) node[black] (v\j) {};
  \path[->,>=latex] (v\j) edge[black,bend left=30,thick] node[black,auto] {$x_{\j}$} (w\j);
  \path[->,>=latex] (w\j) edge[black,bend left=30,thick] node[black,auto] {$y_{\j}$}(v\j);
}
\draw (90:\radius) node[above=3pt] {$5$};
\draw (162:\radius) node[above left] {$4$};
\draw (234:\radius) node[below left] {$3$};
\draw (306:\radius) node[below right] {$2$};
\draw (18:\radius) node[above right] {$1$};
\end{tikzpicture}
\end{center}

The algebra $B_{k,n}$  coincides with the quotient of the completed path
algebra of the graph $C$ (a circular graph with vertices
$C_0=\mathbb Z_n$ set clockwise around a circle, and with the set of edges, $C_1$, also
labeled by $\mathbb Z_n$, with edge $i$ joining vertices $i-1$ and $i$), i.e.\ the doubled quiver as above,
by the closure of the ideal generated by the relations above (we view the completed path
algebra of the graph $C$ 
as a topological algebra via the $m$-adic topology, where $m$ is the two-sided ideal
generated by the arrows of the quiver, see \cite[Section 1]{DWZ08}). The algebra 
$B_{k,n}$, that we will often denote by $B$ when there is no ambiguity, 
was introduced in \cite[Section 3]{JKS16}. 
Observe that $B_{k,n}$ is isomorphic to $B_{n-k,n}$, so we will always take $k\le \frac n 2$. 

\smallskip

The (maximal) Cohen-Macaulay $B$-modules are precisely those which are
free as $Z$-modules. Such a module $M$ is given by a representation
$\{M_i\,:\,i\in C_0\}$ of
the quiver with each $M_i$ a free $Z$-module of the same rank
(which is the rank of $M$).

\begin{defn}[\cite{JKS16}, Definition 3.5]
For any $B_{k,n}$-module $M$ and $K$ the field of fractions of $Z$, the {\bf rank} 
of $M$, denoted by ${\rm rk}(M)$,  is defined 
to be ${\rm rk}(M) = {\rm len}(M \otimes_Z K)$. 
\end{defn}

Note that $B\otimes_Z K\cong M_n ( K)$, 
which is a simple algebra. It is easy to check that the rank is additive on short exact sequences,
that ${\rm rk} (M) = 0$ for any finite-dimensional $B$-module 
(because these are torsion over $Z$) and 
that, for any Cohen-Macaulay $B$-module $M$ and every idempotent $e_j$, $1\leq j\leq n$, ${\rm rk}_Z(e_j M) = {\rm rk}(M)$, so that, in particular, ${\rm rk}_Z(M) = n  {\rm rk}(M)$.

\begin{defn}[\cite{JKS16}, Definition 5.1] \label{d:moduleMI}
For any $k$-subset   $I$  of $C_1$, we define a rank $1$ $B$-module
\[
  L_I = (U_i,\ i\in C_0 \,;\, x_i,y_i,\, i\in C_1)
\]
as follows.
For each vertex $i\in C_0$, set $U_i=\mathbb C[[t]]$ and,
for each edge $i\in C_1$, set
\begin{itemize}
\item[] $x_i\colon U_{i-1}\to U_{i}$ to be multiplication by $1$ if $i\in I$, and by $t$ if $i\not\in I$,
\item[] $y_i\colon U_{i}\to U_{i-1}$ to be multiplication by $t$ if $i\in I$, and by $1$ if $i\not\in I$.
\end{itemize}
\end{defn}

The module $L_I$ can be represented by a lattice diagram
$\mathcal{L}_I$ in which $U_0,U_1,U_2,\ldots, U_n$ are represented by columns of vertices (dots) from
left to right (with $U_0$ and $U_n$ to be identified), going down infinitely. 
The vertices in each column correspond to the natural monomial 
$\mathbb C$-basis of $\mathbb C[t]$.
The column corresponding to $U_{i+1}$ is displaced half a step vertically
downwards (respectively, upwards) in relation to $U_i$ if $i+1\in I$
(respectively, $i+1\not \in I$), and the actions of $x_i$ and $y_i$ are
shown as diagonal arrows. Note that the $k$-subset $I$ can then be read off as
the set of labels on the arrows pointing down to the right which are exposed
to the top of the diagram. For example, the lattice diagram $\mathcal{L}_{\{1,4,5\}}$
in the case $k=3$, $n=8$, is shown in the following picture   
\begin{center}
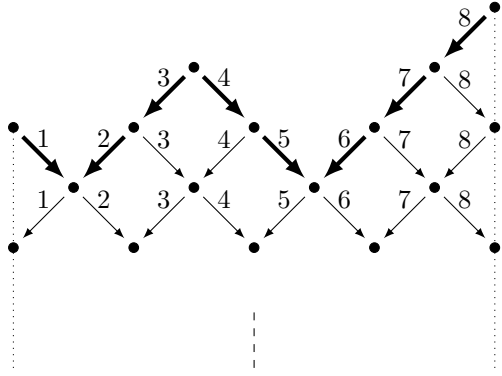
\begin{figure}[H]
\center
\begin{tikzpicture}[scale=0.8,baseline=(bb.base),
quivarrow/.style={black, -latex, thin}]
\newcommand{\seventh}{51.4} 
\newcommand{\circradius}{1.5cm}
\newcommand{\inradius}{1.2cm}
\newcommand{\outradius}{1.8cm}
\newcommand{\dotrad}{0.1cm} 
\newcommand{\bdrydotrad}{{0.8*\dotrad}} 
\path (0,0) node (bb) {}; 


\draw (0,0) circle(\bdrydotrad) [fill=black];
\draw (0,2) circle(\bdrydotrad) [fill=black];
\draw (1,1) circle(\bdrydotrad) [fill=black];
\draw (2,0) circle(\bdrydotrad) [fill=black];
\draw (2,2) circle(\bdrydotrad) [fill=black];
\draw (3,1) circle(\bdrydotrad) [fill=black];
\draw (3,3) circle(\bdrydotrad) [fill=black];
\draw (4,0) circle(\bdrydotrad) [fill=black];
\draw (4,2) circle(\bdrydotrad) [fill=black];
\draw (5,1) circle(\bdrydotrad) [fill=black];
\draw (6,0) circle(\bdrydotrad) [fill=black];
\draw (6,2) circle(\bdrydotrad) [fill=black];
\draw (7,1) circle(\bdrydotrad) [fill=black];
\draw (7,3) circle(\bdrydotrad) [fill=black];
\draw (8,2) circle(\bdrydotrad) [fill=black];
\draw (8,4) circle(\bdrydotrad) [fill=black];
\draw (8,0) circle(\bdrydotrad) [fill=black];


\draw [quivarrow,shorten <=5pt, shorten >=5pt, ultra thick] (0,2)-- node[above]{$1$} (1,1);
\draw [quivarrow,shorten <=5pt, shorten >=5pt] (1,1) -- node[above]{$1$} (0,0);
\draw [quivarrow,shorten <=5pt, shorten >=5pt, ultra thick] (2,2) -- node[above]{$2$} (1,1);
\draw [quivarrow,shorten <=5pt, shorten >=5pt] (1,1) -- node[above]{$2$} (2,0);
\draw [quivarrow,shorten <=5pt, shorten >=5pt, ultra thick] (3,3) -- node[above]{$3$} (2,2);
\draw [quivarrow,shorten <=5pt, shorten >=5pt] (2,2) -- node[above]{$3$} (3,1);
\draw [quivarrow,shorten <=5pt, shorten >=5pt] (3,1) -- node[above]{$3$} (2,0);
\draw [quivarrow,shorten <=5pt, shorten >=5pt, ultra thick] (3,3) -- node[above]{$4$} (4,2);
\draw [quivarrow,shorten <=5pt, shorten >=5pt] (4,2) -- node[above]{$4$} (3,1);
\draw [quivarrow,shorten <=5pt, shorten >=5pt] (3,1) -- node[above]{$4$} (4,0);
\draw [quivarrow,shorten <=5pt, shorten >=5pt, ultra thick] (4,2) -- node[above]{$5$} (5,1);
\draw [quivarrow,shorten <=5pt, shorten >=5pt] (5,1) -- node[above]{$5$} (4,0);
\draw [quivarrow,shorten <=5pt, shorten >=5pt, ultra thick] (6,2) -- node[above]{$6$} (5,1);
\draw [quivarrow,shorten <=5pt, shorten >=5pt] (5,1) -- node[above]{$6$} (6,0);
\draw [quivarrow,shorten <=5pt, shorten >=5pt] (6,2) -- node[above]{$7$} (7,1);
\draw [quivarrow,shorten <=5pt, shorten >=5pt] (7,1) -- node[above]{$7$} (6,0);
\draw [quivarrow,shorten <=5pt, shorten >=5pt, ultra thick] (7,3) -- node[above]{$7$} (6,2);
\draw [quivarrow,shorten <=5pt, shorten >=5pt] (7,3) -- node[above]{$8$} (8,2);
\draw [quivarrow,shorten <=5pt, shorten >=5pt] (8,2) -- node[above]{$8$} (7,1);
\draw [quivarrow,shorten <=5pt, shorten >=5pt, ultra thick] (8,4) -- node[above]{$8$} (7,3);
\draw [quivarrow,shorten <=5pt, shorten >=5pt] (7,1) -- node[above]{$8$} (8,0);

\draw [dotted] (0,-2) -- (0,2);
\draw [dotted] (8,-2) -- (8,4);

\draw [dashed] (4,-2) -- (4,-1);
\end{tikzpicture}
\caption{Lattice diagram of the module $L_{\{1,4,5\}}$} \label{Lattice}
\end{figure}
\end{center}

We see from Figure \ref{Lattice} that the module $L_I$ is determined by its upper boundary, denoted  by the thick lines, 
which we refer to as the {\em rim} of the module $L_I$ (this is why we call the $k$-subset $I$  the rim of $L_I$). 
Throughout this paper we will identify a rank 1 module $L_I$ with its rim. Moreover, most of the time we will omit the arrows in the rim of $L_I$ and represent it as an undirected graph. 

We say that $i$ is a {\em peak} of the rim $I$ if $i\notin I$ and $i+1\in I$.  In the above example, the peaks of $I=\{1,4,5\}$ are $3$ and $8$. We say that $i$ is a {\em valley} of the rim $I$ if $i\in I$ and $i+1\notin I$.  In the above example, the valleys of $I=\{1,4,5\}$ are $1$ and $5$.

\begin{prop}[\cite{JKS16}, Proposition 5.2]
Every rank $1$ Cohen-Macaulay $B$-module is isomorphic to $L_I$
for some unique $k$-subset $I$ of $C_1$.
\end{prop}

Every $B$-module has a canonical endomorphism given by multiplication by $t\in Z$.
For ${L}_I$ this corresponds to shifting $\mathcal{L}_I$ one step downwards.
Since $Z$ is central, ${\rm Hom}_B(M,N)$ is
a $Z$-module for arbitrary $B$-modules $M$ and $N$. If $M,N$ are free $Z$-modules, then so is ${\rm Hom}_B(M,N)$. In particular, for any two rank 1 
Cohen-Macaulay $B$-modules $L_I$ and $L_J$, ${\rm Hom}_B(L_I,L_J)$ is a free 
module of rank 1
over $Z=\mathbb C[[t]]$, generated by the canonical map given by placing the 
lattice of $L_I$ inside the lattice of $L_J$ as far up as possible so that no part of the rim of $L_I$ is strictly above the rim of $L_J$ \cite[Section 6]{JKS16}.

\begin{defn}[$r$-interlacing] \label{interlacing}
Let $I$ and $J$ be two $k$-subsets of $\{1,\dots,n\}$. The sets $I$ and $J$ are said to be {\em $r$-interlacing} if there exist subsets  $\{i_1,i_3,\dots,i_{2r-1}\}\subset I\setminus J$ and $\{i_2,i_4,\dots, i_{2r}\}\subset J\setminus I$ 
such that $i_1<i_2<i_3<\dots <i_{2r}<i_1$ (cyclically) 
and if there exist no larger subsets of $I$ and of $J$ with this property.  We say that $I$ and $J$ are 
{\em tightly $r$-interlacing} if they are $r$-interlacing and $|I\cap J|=k-r.$
\end{defn}

\begin{defn}\label{def-rigid}
A $B$-module is \emph{rigid} if ${\rm Ext}^1_B (M,M)=0$.
\end{defn}

If $I$ and $J$ are $r$-interlacing $k$-subsets, where $r<2$, then ${\rm Ext}_{B}^1(L_I,L_J)=0$, in particular, 
rank 1 modules are rigid (see \cite[Proposition 5.6]{JKS16}).

Every rigid indecomposable $M$ of rank $n$ in ${\rm CM}(B)$ has a filtration having factors 
$L_{I_1},L_{I_{2}},\dots, L_{I_n}$ of rank 1. 
This filtration is noted in its \emph{profile}, 
${\rm pr} (M) = I_1 \mid I_2\mid\ldots \mid I_n$, \cite[Corollary 6.7]{JKS16}.
In the case of  a rank $2$ module $M$ with filtration $L_I\mid L_J$ (i.e.~with profile $I\mid J$), 
we picture this module by drawing the rim $J$ below the rim $I$, in such a way that $J$ is placed as far up as possible so that no part of the rim $J$ is strictly above the rim $I$. We refer to this picture of $M$ 
as its {\em lattice diagram.}
Note that there is at least one point where the rims $I$ and $J$ meet (see {Figure~\ref{fig:3boxes-i-j} for an example}). 

\begin{rem}\label{rem:poset-boxes}
Suppose that the two $k$-subsets $I$ and $J$ are $r$-interlacing and that $M$ is a rank $2$ module 
with profile $I\mid J$. Then the two rims in the lattice diagram of $M$ form a number of regions between 
the points where the two rims meet but differ in direction before and/or after meeting. 
We call these regions the 
{\em boxes} formed by the rims or by the profile. 
The term box is a combinatorial tool which will be very useful in finding conditions for indecomposability.  
Let us point out, however, that the module $M$ might be a direct sum in which case the lattice diagram is really a pair of  lattice diagrams of rank 1 modules. We still view the corresponding diagram as forming boxes. 
If $I$ and $J$ are $r$-interlacing, then they form exactly $r$-boxes if and only if they are tightly $r$-interlacing.
(If we consider 
the lattice diagram as an infinite branched graph in the plane, the boxes are the closures of the 
finite regions in the complement.)
A lattice diagram with three boxes is shown in Figure~\ref{fig:boxes-poset}. 
If $M$ is a rank 2 module with $r_1$ boxes, with $r_1\le r$, the poset structure associated with $M$ is $1^{r_1}\mid 2$, see 
Figure~\ref{fig:boxes-poset}. 
\end{rem}

For background on the poset associated with an indecomposable module or with its profile, 
we refer the reader to~\cite[Section 6]{JKS16} and to~\cite[Section 2]{BBGEL20}.

{Consider the tame cases $(k,n)=(3,9)$ or $(k,n)=(4,8)$ and let $M$ be a rigid indecomposable rank 2 module of 
${\rm CM}(B_{k,n})$. Then $M\cong L_I\mid L_J$ where $I$ and $J$ are $3$-interlacing, \cite[Proposition 5.5]{BBGE}.} 
{Furthermore, we also know that if $I$ and $J$ are tightly $3$-interlacing and if $M\cong L_I\mid L_J$, then 
$M$ is indecomposable, \cite[Lemma 5.11]{BBGE}.} 

We therefore start studying pairs of tightly $3$-interlacing $k$-subsets in order to construct indecomposable 
rank 2 modules and will later consider higher interlacing.

Throughout the paper, our strategy to prove a module is indecomposable is to show that its endomorphism ring does not have 
non-trivial idempotent elements. 
When we deal with a decomposable rank 2 module, in order to determine the summands of this module, we construct a 
non-trivial idempotent in its endomorphism ring, and then find corresponding eigenvectors at each vertex of the quiver and check 
the action of the morphisms $x_i$ on these eigenvectors.

\section{Tight $3$-interlacing} \label{sec:tight-3}

In this section we give a construction of rank 2 indecomposable modules with the profile $I\mid J$ in the case when $I$ and $J$ are  tightly 3-interlacing $k$-subsets, i.e.\ when $|I\setminus J|=|J\setminus I|=3$ 
and non-common elements of $I$ and $J$ interlace. This covers all indecomposable rank 2 modules in the tame case $(3,9)$ and almost all indecomposable rank 2 modules in the tame case $(4,8)$.

We want to define a rank 2 module $\MM(I,J)$ with filtration $L_I\mid L_J$ in a similar way as rank 1 modules 
are defined in ${\rm CM}(B_{k,n})$.   
Let $V_i:=\CC[|t|]\oplus\CC[|t|]$, $i=1,\dots, n$. 
The module $\MM(I,J)$ has $V_i$ at each vertex $1,2,\dots, n$ of $\Gamma_n$.   In order to have a 
module structure for $B_{k,n}$, for every $i$ we need to define  $x_i\colon V_{i-1}\to V_{i}$ and $y_i\colon V_{i}\to V_{i-1}$ in such a way that $x_iy_i=t\cdot \id$ and $x^k=y^{n-k}$.   

Since $L_J$ is a submodule of a rank 2 module $\MM(I,J)$, and $L_I$ is the quotient, if we extend the basis of $L_J$ to the basis of the module $\MM(I,J)$, then with respect to that basis all the matrices $x_i$, $y_i$ must be upper triangular with diagonal entries from the set $\{1,t\}$. More precisely, the diagonal of $x_i$ (resp.\ $y_i$) is $(1,t)$ (resp.\ $(t,1)$) if $i\in J\setminus I$, it is $(t,1)$ (resp.\ $(1,t)$)  if $i\in I\setminus J$, $(t,t)$ (resp.\ $(1,1)$) if $i\in I^c\cap J^c$, and $(1,1)$ (resp.\ $(t,t)$) if $i\in I\cap J$. The only entries in all these matrices that are left to be determined are the ones in the upper right corner. 

Let us assume that $n=6$, $I=\{1,3,5\}$, and $J=\{2,4,6\}$. In the general case, the arguments are the same. 
Denote by $b_i$ the upper right corner element of $x_i$. From  $x_iy_i=t\cdot \id$, we have that the upper right corner 
element of $y_i$ is $-b_i$.  From the relation $x^k=y^{n-k}$ it follows that $\sum_{i=1}^6b_i=0$. 
Thus, our module $\MM(I,J)$ is
\begin{center} \hfil
\xymatrix{ 
0\ar@<.8ex>[rr]^{\begin{pmatrix} t & b_1 \\ 0 & 1 \end{pmatrix}}
&&1 \ar@<.8ex>[rr]^{\begin{pmatrix} 1 & b_2 \\ 0 & t \end{pmatrix}} \ar@<.8ex>[ll]^{\begin{pmatrix} 1 & -b_1 \\ 0 & t \end{pmatrix}} 
&& 2  \ar@<.8ex>[rr]^{\begin{pmatrix} t & b_3 \\ 0 & 1 \end{pmatrix}}\ar@<.8ex>[ll]^{\begin{pmatrix} t & -b_2 \\ 0 & 1 \end{pmatrix}} 
&& 3  \ar@<.8ex>[rr]^{\begin{pmatrix} 1 & b_4 \\ 0 & t \end{pmatrix}} \ar@<.8ex>[ll]^{\begin{pmatrix} 1 & -b_3 \\ 0 & t \end{pmatrix}}
&&4 \ar@<.8ex>[rr]^{\begin{pmatrix} t & b_5 \\ 0 & 1 \end{pmatrix}} \ar@<.8ex>[ll]^{\begin{pmatrix} t & -b_4 \\ 0 & 1 \end{pmatrix}}
&&5 \ar@<.8ex>[rr]^{\begin{pmatrix} 1 & b_6 \\ 0 & t \end{pmatrix}} \ar@<.8ex>[ll]^{\begin{pmatrix} 1 & -b_5 \\ 0 & t \end{pmatrix}}
&&6\,\,\,\ar@<.8ex>[ll]^{\begin{pmatrix} t & -b_6 \\ 0 & 1 \end{pmatrix}}
} \hfil
\end{center}
with $\sum b_i=0$. 
We say that $\mathbb M(I,J)$ is determined by the 6-tuple $(b_1,b_2,b_3, b_4,b_5,b_6)$. 

%
\subsection{Divisibility conditions for (in)decomposability}\label{ssec:divisibility}

Let $I,J$ be two $k$-subsets and $\MM(I,J)$ be given by the tuple $(b_1,b_2,b_3, b_4,b_5,b_6)$ with 
$\sum b_i=0$. 
The question is how to determine the $b_i$'s so that the module $\mathbb M(I,J)$ is indecomposable. Assume first that $\mathbb M(I,J)$ is decomposable and that $L_J$ is a direct summand of $\mathbb M(I,J)$. Then there exists a retraction $\mu=(\mu_i)_{i=1}^6$ such that $\mu_i\circ \theta_i=\id$, where $(\theta_i)_{i=1}^6$ is the natural injection of $L_J$ into $\mathbb M(I,J)$. Using the same basis as before, we can assume that $\mu_i=[1\,\, \alpha_i ]$ for some 
$\alpha_i\in \CC[[t]]$. From the commutativity relations we have $\id\circ \mu_i=\mu_{i+1}\circ x_{i+1}$ for $i$ odd, and  $t\cdot \id\circ \mu_i=\mu_{i+1}\circ x_{i+1}$ for $i$ even. It follows that 
$\alpha_i=b_{i+1}+t\alpha_{i+1}$ for $i$ odd, and $t\alpha_i=b_{i+1}+\alpha_{i+1}$ for $i$ even. From this we have 
\begin{align*}
t(\alpha_2-\alpha_4)&=b_3+b_4,\\
t(\alpha_4-\alpha_6)&=b_5+b_6,\\
t(\alpha_6-\alpha_2)&=b_1+b_2.
\end{align*}
Thus, if $L_J$ is a direct summand of $\mathbb M(I,J)$, then $t| b_3+b_4$, $t|b_5+b_6$, and $t| b_1+b_2$ (and we can easily find elements $\alpha_i$, $i=1,\dots,6$, satisfying previous equations). If only one of these conditions is not met, then $L_J$ is not a direct summand of $M$. For example, if we choose $b_2=-b_3=0$, $b_4=-b_5=1$, and $b_6=-b_1=2$ in the construction of the module  $\mathbb M(I,J)$, then $L_J$ is not a direct summand of $\mathbb M(I,J)$.  
Our aim is to study the structure of the module $\MM(I,J)$ in terms of the divisibility conditions the coefficients $b_i$ satisfy.  

\begin{rem} If $L_J$ is not a summand of $M$, it does not mean that $M$ is indecomposable (cf.\ Theorem 3.12 in \cite{BBGE}). 
\end{rem}

Let us now consider the general case, i.e.\ let  $\mathbb M(I,J)$ be the module as defined above,  but in general terms when $I$ and $J$ are tightly $3$-interlacing. Write $I\setminus J$ as $\{i_1,i_2,i_3\}$ and $J\setminus I=\{j_1,j_2,j_3\}$ so that 
$1\le i_1<j_1<i_2<j_2<i_3<j_3\le n$. Define 
\begin{align*}
&&&&&&x_{i_{r}}&=\begin{pmatrix} t& b_{2r-1} \\ 0 & 1 \end{pmatrix},& x_{j_{r}}&=\begin{pmatrix} 1& b_{2r} \\ 0 & t \end{pmatrix}, &&&&&\\ 
&&&&&&y_{i_{r}}&=\begin{pmatrix} 1& -b_{2r-1} \\ 0 & t \end{pmatrix}, &y_{j_{r}}&=\begin{pmatrix} t& -b_{2r} \\ 0 & 1 \end{pmatrix},&&&&&
\end{align*}
for $r=1,2,3$ (see the previous figure for $n=6$).  For  $i\in  I^c \cap J^c$,  we set  $x_i=\begin{pmatrix} t & 0 \\ 0 & t \end{pmatrix}$ and $y_i=\begin{pmatrix} 1 & 0 \\ 0 & 1 \end{pmatrix}$. For  $i\in I\cap J$,  we set  $x_i=\begin{pmatrix} 1 & 0 \\ 0 & 1 \end{pmatrix}$ and $y_i=\begin{pmatrix} t & 0 \\ 0 & t \end{pmatrix}$.   Also, we assume that $\sum_{r=1}^6b_r=0$.  Note that for  $i\in  (I^c \cap J^c)\cup (I\cap J)$ we define the matrices $x_i$ and $y_i$   to be diagonal, i.e.\ we assume that the upper right corner of $x_i$ and $y_i$ is $0$ if $i\in  (I^c \cap J^c)\cup (I\cap J)$. {We can achieve this under a suitable base change of  $V_i$}.

 By construction it holds that  $xy=yx$ and $x^k=y^{n-k}$ at all vertices, and that $\MM(I,J)$ is free over the centre of the boundary algebra. Hence, the following proposition holds.

\begin{prop} 
The above-constructed module $\MM(I,J)$ is in 
${\rm CM}(B_{k,n})$. 
\end{prop} 

For the remainder of the paper, if $w=t^av$, for some positive integer $a$, then $t^{-a}w$ 
stands for $v$. 

\begin{prop} \label{lm:n6-hom}
Let $I,J$ be tightly $3$-interlacing, $n\ge 6$ arbitrary, $I\setminus J=\{i_1,i_2,i_3\}$, and $J\setminus I=\{j_1,j_2,j_3\}$, where $1\le i_1<j_1<i_2<j_2<i_3<j_3\le n$. If $\varphi=( \varphi_i)_{i=1}^n\in$ {\rm Hom}$(\MM(I,J),\MM(I,J))$, then   
\begin{align*}
\varphi_{j_3}&=\begin{pmatrix}a & b \\ c & d  \end{pmatrix},\\
\varphi_{i_1}&= \begin{pmatrix}a+b_1t^{-1}c & tb+(d-a)b_1-b_1^2t^{-1}c \\ t^{-1}c& d-b_1t^{-1}c  \end{pmatrix},\\
\varphi_{j_1}&=\begin{pmatrix}a+(b_1+b_2)t^{-1}c & b+t^{-1}((d-a)(b_1+b_2)-(b_1+b_2)^2t^{-1}c) \\ c& d-(b_1+b_2)t^{-1}c  \end{pmatrix},\\
\varphi_{i_2}&=\begin{pmatrix}a+(b_1+b_2+b_3)t^{-1}c & tb+(d-a)(b_1+b_2+b_3)-(b_1+b_2+b_3)^2t^{-1}c \\ t^{-1}c& d-(b_1+b_2+b_3)t^{-1}c  \end{pmatrix},\\
\varphi_{j_2}&=\begin{pmatrix}a+(b_1+b_2+b_3+b_4)t^{-1}c & b+t^{-1}((d-a)(b_1+b_2+b_3+b_4)-(b_1+b_2+b_3+b_4)^2t^{-1}c) \\ c& d-(b_1+b_2+b_3+b_4)t^{-1}c  \end{pmatrix},\\
\varphi_{i_3}&=\begin{pmatrix}a-b_6t^{-1}c & tb-(d-a)b_6-b_6^2t^{-1}c \\ t^{-1}c& d+b_6t^{-1}c  \end{pmatrix},\\
\varphi_{i}&=\varphi_{i-1}, \,\, \text{for } i\in (I^c \cap J^c)\cup (I\cap J),
\end{align*}
with $a,b,c,d\in \CC[|t|]$. Furthermore, $t\mid c$, $t\mid (d-a)(b_1+b_2)-(b_1+b_2)^2t^{-1}c$, and  $t\mid (d-a)(b_1+b_2+b_3+b_4)-(b_1+b_2+b_3+b_4)^2t^{-1}c$.  
\end{prop}
\begin{proof} 

First we consider the case $n=6$.  Let $\varphi=(\varphi_1,\dots, \varphi_6)$ be an endomorphism of $\MM(I,J)$, where each $\varphi_i$ is an element of $M_2(\CC[|t|])$ (matrices over the centre). 
\begin{center}
\xymatrix{
6\ar@<.8ex>[rr]^{\begin{pmatrix} t & b_1 \\ 0 & 1 \end{pmatrix}} \ar@<.1ex>[ddd]_{\varphi_6}
&&1 \ar@<.8ex>[rr]^{\begin{pmatrix} 1 & b_2 \\ 0 & t \end{pmatrix}} \ar@<.8ex>[ll]^{\begin{pmatrix} 1 & -b_1 \\ 0 & t \end{pmatrix}} \ar@<.1ex>[ddd]_{\varphi_1}
&& 2  \ar@<.8ex>[rr]^{\begin{pmatrix} t & b_3 \\ 0 & 1 \end{pmatrix}}\ar@<.8ex>[ll]^{\begin{pmatrix} t & -b_2 \\ 0 & 1 \end{pmatrix}} \ar@<.1ex>[ddd]_{\varphi_2}
&& 3  \ar@<.8ex>[rr]^{\begin{pmatrix} 1 & b_4 \\ 0 & t \end{pmatrix}} \ar@<.8ex>[ll]^{\begin{pmatrix} 1 & -b_3 \\ 0 & t \end{pmatrix}} \ar@<.1ex>[ddd]_{\varphi_3}
&&4 \ar@<.8ex>[rr]^{\begin{pmatrix} t & b_5 \\ 0 & 1 \end{pmatrix}} \ar@<.8ex>[ll]^{\begin{pmatrix} t & -b_4 \\ 0 & 1 \end{pmatrix}} \ar@<.1ex>[ddd]_{\varphi_4}
&&5 \ar@<.8ex>[rr]^{\begin{pmatrix} 1 & b_6 \\ 0 & t \end{pmatrix}} \ar@<.8ex>[ll]^{\begin{pmatrix} 1 & -b_5 \\ 0 & t \end{pmatrix}} \ar@<.1ex>[ddd]_{\varphi_5}
&&6 \ar@<.8ex>[ll]^{\begin{pmatrix} t & -b_6 \\ 0 & 1 \end{pmatrix}}\ar@<.1ex>[ddd]_{\varphi_6}\\
&&&&&&&&&&\\
&&&&&&&&&&\\
6\ar@<.8ex>[rr]^{\begin{pmatrix} t & b_1 \\ 0 & 1 \end{pmatrix}}
&&1 \ar@<.8ex>[rr]^{\begin{pmatrix} 1 & b_2 \\ 0 & t \end{pmatrix}} \ar@<.8ex>[ll]^{\begin{pmatrix} 1 & -b_1 \\ 0 & t \end{pmatrix}} 
&& 2  \ar@<.8ex>[rr]^{\begin{pmatrix} t & b_3\\ 0 & 1 \end{pmatrix}}\ar@<.8ex>[ll]^{\begin{pmatrix} t & -b_2\\ 0 & 1 \end{pmatrix}} 
&& 3  \ar@<.8ex>[rr]^{\begin{pmatrix} 1 & b_4 \\ 0 & t \end{pmatrix}} \ar@<.8ex>[ll]^{\begin{pmatrix} 1 & -b_3 \\ 0 & t \end{pmatrix}}
&&4 \ar@<.8ex>[rr]^{\begin{pmatrix} t & b_5 \\ 0 & 1 \end{pmatrix}} \ar@<.8ex>[ll]^{\begin{pmatrix} t & -b_4 \\ 0 & 1 \end{pmatrix}}
&&5 \ar@<.8ex>[rr]^{\begin{pmatrix} 1 & b_6 \\ 0 & t \end{pmatrix}} \ar@<.8ex>[ll]^{\begin{pmatrix} 1 & -b_5 \\ 0 & t \end{pmatrix}}
&&6 \ar@<.8ex>[ll]^{\begin{pmatrix} t & -b_6 \\ 0 & 1 \end{pmatrix}}
}
\end{center}

We use commutativity relations $x_{i+1}\varphi_i = \varphi_{i+1}x_{i+1}$, i.e.\ we check the relations:
\[
\begin{array}{clccl}
(i) & x_2\varphi_1 = \varphi_2x_2, & \quad & (ii) & x_3\varphi_2 =  \varphi_3x_3, \\
(iii) & x_4\varphi_3  = \varphi_4x_4, & & (iv) & x_5\varphi_4 = \varphi_5x_5,  \\
(v) & x_6\varphi_5   = \varphi_6x_6, & & (vi) & x_1\varphi_6   =  \varphi_1 x_1.
\end{array}
\]

Let $\varphi_0=\begin{pmatrix} a & b \\ c & d \end{pmatrix}$ and 
$\varphi_1=\begin{pmatrix} e & f \\ g & h \end{pmatrix}$.  From (vi), we get $\begin{pmatrix} at+b_1c & bt+b_1d \\ c & d\end{pmatrix} = \begin{pmatrix} et & eb_1+f \\ gt & h+gb_1 \end{pmatrix},$
and  $at+b_1c=et$, $bt+b_1d=eb_1+f$, $c=gt$, and $d=gb_1+h$. It follows that $t\mid c$, and that 
$$\varphi_1=\begin{pmatrix}a+b_1t^{-1}c & tb+(d-a)b_1-b_1^2t^{-1}c \\ t^{-1}c& d-b_1t^{-1}c  \end{pmatrix}.$$

Similarly, if   $\varphi_2=\begin{pmatrix} e & f \\ g & h \end{pmatrix}$, then equality $x_2\varphi_1 = \varphi_2x_2$ yields $$\varphi_2=\begin{pmatrix}a+(b_1+b_2)t^{-1}c & b+t^{-1}((d-a)(b_1+b_2)-(b_1+b_2)^2t^{-1}c) \\ c& d-(b_1+b_2)t^{-1}c  \end{pmatrix},$$ and $t\mid (d-a)(b_1+b_2)-(b_1+b_2)^2t^{-1}c$.  

The rest of the proof for $\varphi_3$, $\varphi_4$, and $\varphi_5$ is analogous. We omit the details of elementary, but tedious computation.

In the general case of arbitrary $n$, the proof is almost the same as for $n=6$. The only thing left to note is that if $i\in (I^c \cap J^c)\cup (I\cap J)$, then $x_i$ is a scalar matrix (either identity or $t$ times identity), so from  $x_{i}\varphi_{i-1}=\varphi_{i}x_{i}$, it follows immediately that  $\varphi_{i-1}=\varphi_{i}$.
\end{proof}

\begin{rem}
Take $\varphi$ as in Proposition~\ref{lm:n6-hom}. The morphism $\varphi$ also satisfies the other six relations $\varphi_i y_{i+1}=y_{i+1}\varphi_{i+1}$. 
Indeed, if $x_{i+1}\varphi_i=\varphi_{i+1}x_{i+1}$, then if we multiply this equality by $y_{i+1}$ both from the left and right, we obtain $t \cdot \varphi_i y_{i+1}=t\cdot y_{i+1}\varphi_{i+1} $. Since $t$ is a regular element in $\mathbb C[[t]]$, after cancellation by $t$ we obtain $\varphi_i y_{i+1}=y_{i+1}\varphi_{i+1}$. 
Also, note that in order to prove that $\varphi$ is idempotent we only need to make sure that only one $\varphi_i$ is idempotent. If $\varphi_i$ is idempotent, then from $x_{i+1}\varphi_i=\varphi_{i+1}x_{i+1}$ when we multiply by $\varphi_{i+1}$ from the left, we have $\varphi_{i+1}x_{i+1}\varphi_i=\varphi^2_{i+1}x_{i+1}$, 
and multiplying by $\varphi_i$ from the right we get 
$\varphi_{i+1}x_{i+1}\varphi_i=x_{i+1}\varphi_i^2$. 
Then  $x_{i+1}\varphi^2_i=\varphi^2_{i+1}x_{i+1}$ and $x_{i+1}\varphi_i=\varphi^2_{i+1}x_{i+1}$ 
as $\varphi_i$ is idempotent, and subsequently, $\varphi_{i+1}x_{i+1}=\varphi^2_{i+1}x_{i+1}$, yielding $\varphi_{i+1}=\varphi^2_{i+1}$ after multiplication by $y_{i+1}$ from the right and cancellation by $t$. 
\end{rem}

We now give necessary and sufficient conditions for the module $\mathbb M(I,J)$ to be indecomposable. 

\begin{theorem} \label{t6}
{Let $I$ and $J$ be tightly 3-interlacing.}
The module $\MM(I,J)$ is indecomposable if and only if  $t\nmid b_{1}+b_{2}$, $t\nmid b_{3}+b_{4}$, and $t\nmid b_{5}+b_{6}$. 
\end{theorem}
\begin{proof}
As before, it is sufficient to consider the case $n=6$, so we can assume that $I=\{1,3,5\}$ and $J=\{2,4,6
\}$. Let $\varphi=( \varphi_i)_{i=1}^6\in$ Hom$(\MM(I,J),\MM(I,J))$ be an idempotent homomorphism and assume that $\varphi_{0}=\begin{pmatrix}a & b \\ c & d  \end{pmatrix}$. From the previous proposition we know that 
\begin{align}
t&\mid c,  \label{eq:1}\\
  t&\mid (d-a)(b_1+b_2)-(b_1+b_2)^2t^{-1}c, \label{eq:2}\\
  t&\mid (d-a)(b_1+b_2+b_3+b_4)-(b_1+b_2+b_3+b_4)^2t^{-1}c. \label{eq:3}
\end{align} 
Assume that $t\nmid b_1+b_2$, $t\nmid b_3+b_4$, and $t\nmid b_5+b_6$. Since $t\nmid b_1+b_2$, it follows from relation (\ref{eq:2}) that $$t\mid d-a-(b_1+b_2)t^{-1}c.$$ Similarly, since $t\nmid b_5+b_6=-(b_1+b_2+b_3+b_4)$, from relation (\ref{eq:3}) follows that  $$t\mid d-a-(b_1+b_2+b_3+b_4)t^{-1}c.$$ Thus, it must hold that $$t\mid (b_3+b_4)t^{-1}c,$$ and since $t\nmid b_3+b_4$, it must be that $t\mid t^{-1}c$, and subsequently that $t\mid d-a$.  

From the fact that $\varphi_0$ is idempotent and $t\mid c$ it follows that $t\mid a-a^2$ and $t\mid d-d^2$. Also, from $\varphi_0^2=\varphi_0$ it follows that either $a=d$ or $a+d=1$.  If $a=d$, then $b=c=0$ (otherwise $a=d=\frac 12$ and $\frac 14=bc$, which is not possible as $c$ is divisible by $t$), and $a=d=1$ or $a=d=0$ giving us the trivial idempotents. If $a+d=1$, then $t\mid a$ or $t\mid d$. Taking into account that $t\mid d-a$, we conclude that $t\mid a$ and $t\mid d$. This implies that $1=a+d$ is divisible by $t$, which is not true. Thus, the only idempotent homomorphisms of $\MM(I,J)$ are the trivial ones. Hence, $\MM(I,J)$ is indecomposable. 

Assume that $t$ divides at least one of the elements $b_1+b_2$, $b_3+b_4$, $b_5+b_6$. If $t$ divides all three of them, we have seen before that $\MM(I,J)$ is the direct sum of $L_I$ and $L_J$, hence it is a decomposable module. The remaining case is when one of $b_1+b_2$, $b_3+b_4$, $b_5+b_6$ is divisible by $t$ and the other two are not divisible by $t$. Note that it is not possible that two of them are divisible by $t$, and one is not, because they sum up to zero. 

Let $t\mid b_5+b_6$, $t\nmid b_1+b_2$, and $t\nmid b_3+b_4$. In order to find a non-trivial idempotent $\varphi$, note that the relation (\ref{eq:3}) holds because $t\mid b_5+b_6=-(b_1+b_2+b_3+b_4)$. Hence, we only need to find elements $a$, $b$, $c$, and $d$ in such a way that $t\mid c$ and $t\mid d-a-(b_1+b_2)t^{-1}c$. Recall that if $a=d$, then we only obtain the trivial idempotents because $t\mid c$. So it must be $a+d=1$ if we want to find a non-trivial idempotent. If we choose $a=1$, $d=0$, then $t\mid 1+(b_1+b_2)t^{-1}c$. Thus   $(b_1+b_2)t^{-1}c=-1+tg,$  for some $g$. We can take $g=0$, i.e.\ $c=t(b_1+b_2)^{-1}$ (recall that $b_1+b_2$ is invertible because $t\nmid b_1+b_2$) giving us the idempotent ($b=0$ since $a-a^2=bc$ and $c\neq 0$)  $$\varphi_0=\begin{pmatrix}
1& 0\\
-t(b_1+b_2)^{-1}&0
\end{pmatrix}.$$ 
 Its orthogonal complement is the idempotent
 $$
 \begin{pmatrix}
0& 0\\
t(b_1+b_2)^{-1}&1
\end{pmatrix}.$$ 
Since these are non-trivial idempotents, it follows that the module $\MM(I,J)$ is decomposable, what we needed to prove. 
\end{proof}

\begin{ex}\label{ex:rank2} 
Let $k=3$, $n=6$, 
$I=\{1,3,5\}$ and $J=\{2,4,6\}$. If $b_1=-2$, $b_2=0$, $b_3=0$, $b_4=1$, $b_5=-1$, and $b_6=2$, then 
it is easily checked that $t\nmid b_i+b_{i+1}$, $i=1,3,5$, and that $\sum_{i=1}^6b_i=0$, thus giving us an indecomposable   {$B_{k,n}$-}module: 
\begin{center}\hfil
\xymatrix{
6\ar@<.8ex>[rr]^{\begin{pmatrix} t & -2 \\ 0 & 1 \end{pmatrix}}
&&1 \ar@<.8ex>[rr]^{\begin{pmatrix} 1 & 0 \\ 0 & t \end{pmatrix}} \ar@<.8ex>[ll]^{\begin{pmatrix} 1 & 2 \\ 0 & t \end{pmatrix}} 
&& 2  \ar@<.8ex>[rr]^{\begin{pmatrix} t & 0 \\ 0 & 1 \end{pmatrix}}\ar@<.8ex>[ll]^{\begin{pmatrix} t & 0 \\ 0 & 1 \end{pmatrix}} 
&& 3  \ar@<.8ex>[rr]^{\begin{pmatrix} 1 & 1 \\ 0 & t \end{pmatrix}} \ar@<.8ex>[ll]^{\begin{pmatrix} 1 & 0 \\ 0 & t \end{pmatrix}}
&&4 \ar@<.8ex>[rr]^{\begin{pmatrix} t & -1 \\ 0 & 1 \end{pmatrix}} \ar@<.8ex>[ll]^{\begin{pmatrix} t & -1 \\ 0 & 1 \end{pmatrix}}
&&5 \ar@<.8ex>[rr]^{\begin{pmatrix} 1 & 2 \\ 0 & t \end{pmatrix}} \ar@<.8ex>[ll]^{\begin{pmatrix} 1 & 1 \\ 0 & t \end{pmatrix}}
&&6\,\,\, \ar@<.8ex>[ll]^{\begin{pmatrix} t & -2 \\ 0 & 1 \end{pmatrix}}
}
\hfill
\end{center}

The lattice diagram (showing only the rims) of $M=L_{135}\mid L_{246}$ is 
\[
\includegraphics[width=6cm]{135-246}
\]
\end{ex}

\begin{prop}\label{sumica} If $t\nmid b_{1}+b_{2}$, $t\nmid b_{3}+b_{4}$, and $t\mid b_{5}+b_{6}$, then $\MM(I,J)$ is isomorphic to $L_{\{i_1,j_1,i_3\}\cup(I\cap J)}\oplus L_{\{i_2,j_2, j_3\}\cup(I\cap J)}.$ 
\end{prop}
\begin{proof} In the proof of the previous theorem we constructed a non-trivial idempotent endomorphism $$\varphi_0=\begin{pmatrix}
1& 0\\
-t(b_1+b_2)^{-1}&0
\end{pmatrix}.$$ 
 It remains to show that $\MM(I,J)\cong L_{\{i_1,j_1,i_3\}}\oplus L_{\{i_2,j_2, j_3\}}.$  We know that $\MM(I,J)$ is the direct sum of rank 1 modules $L_X$ and $L_Y$ for some $X$ and $Y$. Let us determine $X$ and $Y$. For this, we take, at vertex $i$,  eigenvectors $v_i$ and $w_i$  corresponding to the eigenvalue $1$ of  the idempotents $\varphi_i$ and $1-\varphi_i$ respectively.  For example, $v_0=[1\,\, ,\,\, -t(b_1+b_2)^{-1}]^t$, $w_0=[0\,\, ,\,\, 1]^t$, $v_1= [ 1-b_1(b_1+b_2)^{-1}  \,\, ,\,\, -(b_1+b_2)^{-1} ]^t$ and  $w_1= [b_1  \,\, , \,\, 1  ]^t$, and so on. A basis for $L_X$ is $\{v_i \mid i=0,\ldots, 5\}$, and a basis for $L_Y$ is $\{w_i \mid i=0,\ldots, 5\}$. Direct computation gives us that $x_1v_0=tv_1$, $x_2v_1=tv_2$, $x_3v_2=v_3$, $x_4v_3=v_4$, $x_5v_4=tv_5$, and $x_6v_5=v_0$. Thus, $X=\{3,4,6\}$. Analogously, $Y=\{1,2,5\}.$  In the general case, it is easily seen that this means that $X=\{i_2,j_2,j_3\}\cup(I\cap J)$ and $Y=\{i_1,j_1,i_3\}\cup(I\cap J)$ because $x_i$ is a scalar matrix for $i\in (I\cap J)\cup(I^c\cap J^c)$.
\end{proof}

\begin{rem} If $n=6$, in the case when $t\nmid b_1+b_2$, $t\nmid b_3+b_4$, and $t\mid b_5+b_6$, if we just rename the vertices of the quiver by adding 2  (modulo $n$) to every vertex or by adding 4 to every vertex, we obtain two modules that correspond to the cases when $t\mid b_1+b_2$, $t\nmid b_3+b_4$, $t\nmid b_5+b_6$  and $t\nmid b_1+b_2$, $t\mid b_3+b_4$, $t\nmid b_5+b_6$. In the general case, these two modules are direct sums $L_{\{i_1, i_2,j_2\}\cup(I\cap J)}\oplus L_{\{j_1, i_3,j_3\}\cup(I\cap J)}$  and $L_{\{i_2, i_3,j_3\}\cup(I\cap J)}\oplus L_{\{i_1,j_1, j_2\}\cup(I\cap J)}.$ 
\end{rem}

\begin{ex}\label{ex:3-6-through}
When $n=6$, $I=\{1,3,5\}$, and $J=\{2,4,6\}$, an indecomposable module which has $L_J$ as a submodule and $L_I$ as a quotient module is given in Example \ref{ex:rank2}. 
Also, there are four different decomposable modules appearing as the middle term in a short exact sequence that has  $L_I$ (as a quotient) and $L_J$ (as a submodule) as end terms: 
\begin{align*}
0\longrightarrow L_J\longrightarrow L_{\{1, 3, 5\}}\oplus L_{\{2,4, 6\}} \longrightarrow L_I \longrightarrow 0,\\
0\longrightarrow L_J\longrightarrow L_{\{1, 2, 5\}}\oplus L_{\{3,4, 6\}} \longrightarrow L_I \longrightarrow 0,\\
0\longrightarrow L_J\longrightarrow L_{\{1, 3, 4\}}\oplus L_{\{2,5, 6\}} \longrightarrow L_I \longrightarrow 0,\\
0\longrightarrow L_J\longrightarrow L_{\{1, 2, 4\}}\oplus L_{\{3,5, 6\}} \longrightarrow L_I \longrightarrow 0.
\end{align*}
The profiles of the four modules that appear in the middle in these short exact sequences are illustrated in 
Figure~\ref{ex:example2}.  {Note that the two rims now stand for two rank 1 modules which are direct summands 
of the module. The pictures each show two lattice diagrams which are overlaid so we can compare the positions 
of the peaks. In particular, the two lattice diagrams in (a) look like the lattice diagram of the indecomposable extension 
described in Example~\ref{ex:rank2}.}

\begin{figure}[H]
\begin{center}
\subfloat[$L_{\{1, 3, 5\}}\oplus L_{\{2,4, 6\}}$]{\includegraphics[width = 6cm]{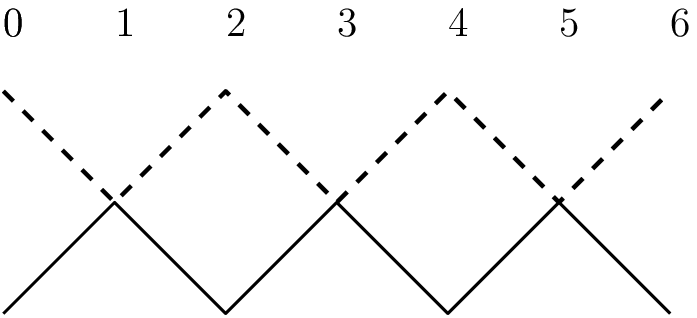}}  \quad \quad \quad \quad
\subfloat[$L_{\{1, 2, 5\}}\oplus L_{\{3,4, 6\}}$]{\includegraphics[width = 6cm ]{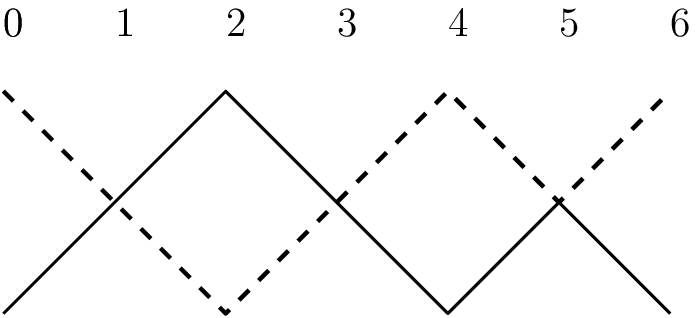}}\\
\subfloat[$L_{\{1, 3, 4\}}\oplus L_{\{2,5, 6\}}$]{\includegraphics[width = 6cm ]{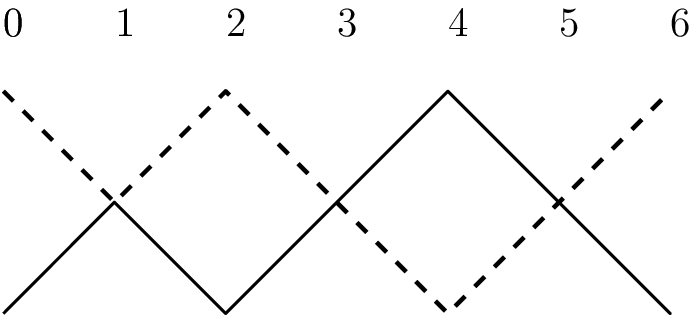}}  \quad \quad \quad \quad
\subfloat[$L_{\{1, 2, 4\}}\oplus L_{\{3,5, 6\}}$]{\includegraphics[width = 6cm ]{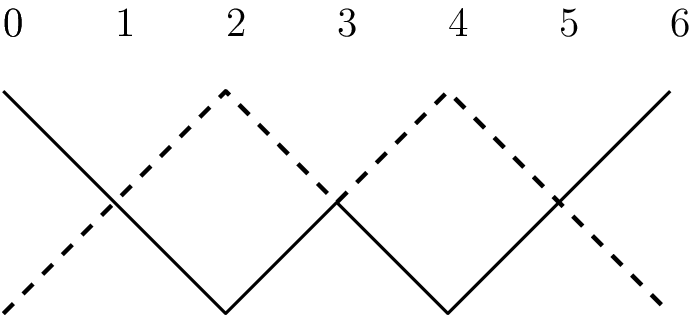}} 
\caption{The {pairs of lattice diagrams} 
of decomposable extensions between $ L_{\{1, 3, 5\}}$ and $L_{\{2,4, 6\}}$.}
\label{ex:example2}
\end{center}
\end{figure}
\end{ex}

\begin{rem} 
Note that for $I=\{1,3,5\}$ and $J=\{2,4,6\}$ there is a non-trivial extension 
$$
\xymatrix{0\ar@{->}[r]&L_J\ar@{->}[rr]^{[\id, -f_1]^t\,\,\,\,\,\,\,\,\,\,\,\,\,\,\,\,\,}&& L_{\{1, 3, 5\}}\oplus L_{\{2,4, 6\}} \ar@{->}[rr]^{\,\,\,\,\,\,\,\,\,\,\,\,\,\,\,\,\,\,\,\,\,\,\,\,\,[f_2, f_3]} && L_I \ar@{->}[r]& 0,}
$$ where $f_i$, $i=1,2,3$,  is the canonical map between rank $1$ modules 
{(see Section~\ref{sec:prelims}, this is a homomorphism of minimal codimension)}. 
The middle term is equal to the direct sum of the end terms, but the maps make this short exact sequence not isomorphic to the trivial sequence. 
\end{rem}

Now we turn our attention to the question of uniqueness of the constructed indecomposable module. If we choose a different set of values for $b_i$, i.e.\ we choose a  6-tuple different from $(-2,0,0,1,-1,2)$,  so that  the conditions $t\nmid  b_3+b_4$, $t\nmid b_5+b_6$, and $t\nmid b_1+b_2$ are fulfilled, then we obtain a module which is not the same as the above constructed module $\mathbb M(I,J)$ in Example \ref{ex:rank2}. In the next theorem, we will  show  directly that for different choices of 6-tuples giving us indecomposable modules with the same filtration $L_I \mid L_J$ we obtain isomorphic modules. Thus, there is a unique indecomposable module with filtration $L_I \mid L_J$ .
\begin{theorem} \label{tiso}
Let $(b_{1}, b_{2}, b_{3}, b_{4}, b_{5}, b_{6} )$ and $(c_{1}, c_{2}, c_{3}, c_{4}, c_{5}, c_{6} )$ be different $6$-tuples corresponding to indecomposable modules $M_1$ and $M_2$, respectively, as constructed 
{in Theorem~\ref{t6}}. 
Then the modules $M_1$ and $M_2$ are isomorphic. 
\end{theorem}
\begin{proof}
As before, it is sufficient to consider the $n=6$ case. We will explicitly construct an isomorphism $\varphi=(\varphi_i)_{i=0}^5$ between the two modules, where $\varphi_i: V_i\longrightarrow W_i$, and $V_i$ and $W_i$ are the vector spaces at vertex $i$ of the modules $M_1$ and $M_2$ respectively.  Also, each $\varphi_i$ is invertible. Recall that $b_i$ (resp.\ $c_i$) is the right upper corner element of $x_i$ for the module $M_1$ (resp.\ $M_2$): 
\begin{center}
\xymatrix{
0\ar@<.8ex>[rr]^{\begin{pmatrix} t & b_1 \\ 0 & 1 \end{pmatrix}} \ar@<.1ex>[ddd]_{\varphi_0}
&&1 \ar@<.8ex>[rr]^{\begin{pmatrix} 1 & b_2 \\ 0 & t \end{pmatrix}} \ar@<.8ex>[ll]^{\begin{pmatrix} 1 & -b_1 \\ 0 & t \end{pmatrix}} \ar@<.1ex>[ddd]_{\varphi_1}
&& 2  \ar@<.8ex>[rr]^{\begin{pmatrix} t & b_3 \\ 0 & 1 \end{pmatrix}}\ar@<.8ex>[ll]^{\begin{pmatrix} t & -b_2 \\ 0 & 1 \end{pmatrix}} \ar@<.1ex>[ddd]_{\varphi_2}
&& 3  \ar@<.8ex>[rr]^{\begin{pmatrix} 1 & b_4 \\ 0 & t \end{pmatrix}} \ar@<.8ex>[ll]^{\begin{pmatrix} 1 & -b_3 \\ 0 & t \end{pmatrix}} \ar@<.1ex>[ddd]_{\varphi_3}
&&4 \ar@<.8ex>[rr]^{\begin{pmatrix} t & b_5 \\ 0 & 1 \end{pmatrix}} \ar@<.8ex>[ll]^{\begin{pmatrix} t & -b_4 \\ 0 & 1 \end{pmatrix}} \ar@<.1ex>[ddd]_{\varphi_4}
&&5 \ar@<.8ex>[rr]^{\begin{pmatrix} 1 & b_6 \\ 0 & t \end{pmatrix}} \ar@<.8ex>[ll]^{\begin{pmatrix} 1 & -b_5 \\ 0 & t \end{pmatrix}} \ar@<.1ex>[ddd]_{\varphi_5}
&&6 \ar@<.8ex>[ll]^{\begin{pmatrix} t & -b_6 \\ 0 & 1 \end{pmatrix}}\ar@<.1ex>[ddd]_{\varphi_0}\\
&&&&&&&&&&\\
&&&&&&&&&&\\
0\ar@<.8ex>[rr]^{\begin{pmatrix} t & c_1 \\ 0 & 1 \end{pmatrix}}
&&1 \ar@<.8ex>[rr]^{\begin{pmatrix} 1 & c_2 \\ 0 & t \end{pmatrix}} \ar@<.8ex>[ll]^{\begin{pmatrix} 1 & -c_1 \\ 0 & t \end{pmatrix}} 
&& 2  \ar@<.8ex>[rr]^{\begin{pmatrix} t & c_3 \\ 0 & 1 \end{pmatrix}}\ar@<.8ex>[ll]^{\begin{pmatrix} t & -c_2 \\ 0 & 1 \end{pmatrix}} 
&& 3  \ar@<.8ex>[rr]^{\begin{pmatrix} 1 & c_4 \\ 0 & t \end{pmatrix}} \ar@<.8ex>[ll]^{\begin{pmatrix} 1 & -c_3 \\ 0 & t \end{pmatrix}}
&&4 \ar@<.8ex>[rr]^{\begin{pmatrix} t & c_5 \\ 0 & 1 \end{pmatrix}} \ar@<.8ex>[ll]^{\begin{pmatrix} t & -c_4 \\ 0 & 1 \end{pmatrix}}
&&5 \ar@<.8ex>[rr]^{\begin{pmatrix} 1 & c_6 \\ 0 & t \end{pmatrix}} \ar@<.8ex>[ll]^{\begin{pmatrix} 1 & -c_5 \\ 0 & t \end{pmatrix}}
&&0 \ar@<.8ex>[ll]^{\begin{pmatrix} t & -c_6 \\ 0 & 1 \end{pmatrix}}
}
\end{center}

Let us assume that $\varphi_i=\begin{pmatrix}\alpha_i & \beta_i \\ \gamma_i & \delta_i  \end{pmatrix}$, for $i=0,\ldots, 5$.  Then from $\varphi_1 \begin{pmatrix} t & b_1 \\  0 & 1  \end{pmatrix}=\begin{pmatrix}t & c_1 \\  0 & 1  \end{pmatrix} \varphi_0$, we obtain $t\mid \gamma_0$, $t \gamma_1= \gamma_0$, $\alpha_1=\alpha_0+c_1t^{-1}\gamma_0$, $\beta_1=\beta_0t-\alpha_0b_1+c_1\delta_0-b_1c_1t^{-1}\gamma_0$, and $\delta_1=\delta_0-b_1t^{-1}\gamma_0$. Hence, $$\varphi_1=\begin{pmatrix} \alpha_0+c_1t^{-1}\gamma_0 & \beta_0t-\alpha_0b_1+c_1\delta_0-b_1c_1t^{-1}\gamma_0 \\ t^{-1}\gamma_0 & \delta_0-b_1t^{-1}\gamma_0 \end{pmatrix}.$$

Since $t\mid \gamma_0$ and we would like $\varphi_0$ to be invertible, then it must be that $t\nmid \alpha_0$ and $t\nmid \delta_0$. Then the inverse of $\varphi_0$ is $\frac{1}{\alpha_0\delta_0-\beta_0\gamma_0}\begin{pmatrix} \delta_0&-\beta_0 \\  -\gamma_0 & \alpha_0 \end{pmatrix}.$
From $\varphi_2 \begin{pmatrix} 1 & b_2 \\  0 & t  \end{pmatrix}=\begin{pmatrix}1 & c_2 \\  0 & t  \end{pmatrix} \varphi_1$, we obtain that $$\varphi_2=\begin{pmatrix} \alpha_0+(c_1+c_2)t^{-1}\gamma_0 & \beta_0+t^{-1}(-\alpha_0(b_1+b_2)+(c_1+c_2)\delta_0 -(b_1+b_2)(c_1+c_2)t^{-1}\gamma_0) \\  \gamma_0 & \delta_0-(b_1+b_2)t^{-1}\gamma_0 \end{pmatrix},$$ 
where $t\mid -\alpha_0(b_1+b_2)+(c_1+c_2)\delta_0 -(b_1+b_2)(c_1+c_2)t^{-1}\gamma_0.$

Analogously, it is easily computed that 
\begin{align*}
\varphi_3&=\begin{pmatrix} \alpha_0+(c_1+c_2+c_3)t^{-1}\gamma_0 & \beta_0t-\alpha_0(\sum_{i=1}^3b_i)+(\sum_{i=1}^3c_i)\delta_0-(\sum_{i=1}^3b_i)(\sum_{i=1}^3c_i)t^{-1}\gamma_0 \\ t^{-1}\gamma_0 & \delta_0-(b_1+b_2+b_3)t^{-1}\gamma_0 \end{pmatrix},\\
\varphi_4&=\begin{pmatrix} \alpha_0-(c_5+c_6)t^{-1}\gamma_0 & \beta_0+t^{-1}(\alpha_0(b_5+b_6)-(c_5+c_6)\delta_0 -(b_5+b_6)(c_5+c_6)t^{-1}\gamma_0) \\  \gamma_0 & \delta_0+(b_5+b_6)t^{-1}\gamma_0 \end{pmatrix},\\ 
\varphi_5&=\begin{pmatrix} \alpha_0-c_6t^{-1}\gamma_0 & \beta_0t+\alpha_0b_6-c_6\delta_0-b_6c_6t^{-1}\gamma_0 \\ t^{-1}\gamma_0 & \delta_0+b_6t^{-1}\gamma_0 \end{pmatrix},
\end{align*}
where $t\mid \alpha_0(b_5+b_6)-(c_5+c_6)\delta_0 -(b_5+b_6)(c_5+c_6)t^{-1}\gamma_0.$

In order to find an isomorphism $\varphi$, we must determine $\alpha_0$, $\beta_0$, $\gamma_0$, and $\delta_0$ satisfying the following conditions: $t\mid \gamma_0$, $t\nmid \alpha_0$, $t\nmid \delta_0$, and 
\begin{align*}
t&\mid -\alpha_0(b_1+b_2)+(c_1+c_2)\delta_0 -(b_1+b_2)(c_1+c_2)t^{-1}\gamma_0,\\
t&\mid\alpha_0(b_5+b_6)-(c_5+c_6)\delta_0 -(b_5+b_6)(c_5+c_6)t^{-1}\gamma_0.
\end{align*}
Note that  there are no conditions attached to $\beta_0$ so we set it to be 0.  If we set 
\begin{align*}
-\alpha_0(b_1+b_2)+(c_1+c_2)\delta_0 -(b_1+b_2)(c_1+c_2)t^{-1}\gamma_0&=0,\\
\alpha_0(b_5+b_6)-(c_5+c_6)\delta_0 -(b_5+b_6)(c_5+c_6)t^{-1}\gamma_0&=0,
\end{align*}
then we get 
$$\alpha_0(b_5+b_6)\left[1+\frac{c_5+c_6}{c_1+c_2} \right]-\delta_0(c_5+c_6)\left[1+\frac{b_5+b_6}{b_1+b_2} \right]=0.$$

If $t\mid 1+\frac{c_5+c_6}{c_1+c_2}$, then from  $\sum_{i=1}^6c_i=0$, we get $t\mid c_3+c_4$, which is not true.  The same holds for $1+\frac{b_5+b_6}{b_1+b_2}$, so both of these elements are invertible. Thus, if we set $\delta_0=1$, then we get $$\alpha_0=\frac{(c_1+c_2)(b_3+b_4)(c_5+c_6)}{(b_1+b_2)(c_3+c_4)(b_5+b_6)},$$ and 
$$\gamma_0=t \frac{(c_3+c_4)(b_5+b_6)-(b_3+b_4)(c_5+c_6)}{(b_1+b_2)(c_3+c_4)(b_5+b_6)}.$$
Hence, 
{\renewcommand\arraystretch{2.5}
$$\varphi_0=\begin{pmatrix} \displaystyle\frac{(c_1+c_2)(b_3+b_4)(c_5+c_6)}{(b_1+b_2)(c_3+c_4)(b_5+b_6)} & \,\,\,\,\,\,\,\, 0 \\  t \displaystyle\frac{(c_3+c_4)(b_5+b_6)-(b_3+b_4)(c_5+c_6)}{(b_1+b_2)(c_3+c_4)(b_5+b_6)} & \,\,\,\,\,\,\,\,1  \end{pmatrix}.$$ 
} 
The other invertible matrices $\varphi_i$ are now determined from the above equalities. Note that all of them are invertible because their determinant is equal to $\alpha_0\delta_0-\beta_0\gamma_0$ which is an invertible element.  Also, $\beta_0=\beta_2=\beta_4=0,$ and $\gamma_0=\gamma_2=\gamma_4=t\gamma_1=t\gamma_3=t\gamma_5.$
\end{proof}

\begin{ex} {\renewcommand\arraystretch{1.1} Assume that $n=6$. We use the previous theorem to construct an isomorphism between the modules corresponding to the $6$-tuples $(b_1, b_2, b_3, b_4, b_5, b_6 )=(-2,0,0,1,-1,2)$ and $(c_1, c_2, c_3, c_4, c_5, c_6 )=(0,1,-1,2,-2,0)$. From the previous theorem we get that $\varphi_0=\begin{pmatrix} 1 & 0 \\  -\frac32t & 1  \end{pmatrix}$, $\varphi_1=\varphi_3=\varphi_5=\begin{pmatrix} 1 & 2 \\  -\frac32 & -2\end{pmatrix}$, $\varphi_2=\begin{pmatrix} -\frac12 & 0 \\  -\frac32 t & -2  \end{pmatrix}$,  $\varphi_4=\begin{pmatrix} -2 & 0 \\  -\frac32t & -\frac12  \end{pmatrix}$. 
}
\end{ex}

\begin{rem} In the case when $t\nmid b_{1}+b_{2}$, $t\nmid  b_{3}+b_{4}$, and $t\mid b_{5}+b_{6}$, we have seen that the module in question is isomorphic to the direct sum $L_{\{i_1,j_1,i_3\}\cup(I\cap J)}\oplus L_{\{i_2,j_2, j_3\}\cup(I\cap J)}$.  This means that regardless of the choice of the elements $b_i$ that satisfy these conditions, we get {a module} 
that is isomorphic to the same direct sum of rank 1 modules. Obviously, the same holds when $t\mid b_{1}+b_{2}$, $t\mid  b_{3}+b_{4}$, and $t\mid b_{5}+b_{6}$, in which case we get the direct sum $L_I\oplus L_J$. Thus, once we know which divisibility conditions our coefficients fulfil, we immediately know which module we are dealing with. 
\end{rem}

\section{Almost tight $3$-interlacing}\label{sec:non-tight3}

In the tame case $(4,8)$, besides the indecomposable modules of rank 2 that we have already 
constructed, i.e.\ the modules whose layers $I$ and $J$ {are 3-interlacing and} satisfy the condition 
$|I\cap J|=k-3$, there {are also the cases of non-tightly $3$-interlacing layers with poset $1^3\mid 2$ and of 
4-interlacing layers (with poset $1^4\mid 2$). In this section, we deal 
with the former case. Recall that the two rims form three boxes in this case (Remark~\ref{rem:poset-boxes}). 
This happens exactly for pairs of subsets $J=\{i,i+2,i+4,i+5\}$, $I=\{i+1,i+3,i+6,i+7\}$, e.g.\ for 
the profile $2478\mid 1356$.} We prove that there is a unique indecomposable rank 2 module with filtration $I\mid J$. We will work in a more general setup which we explain now. 
 
 So we assume for now, for general $k$ and $n$, that $I$ and $J$ are $r$-interlacing for some $r\ge 3$ and that $I\mid J$ forms three boxes 
with poset $1^3\mid 2$.   
\begin{figure}[H]
\begin{center}
{\includegraphics[width = 8cm]{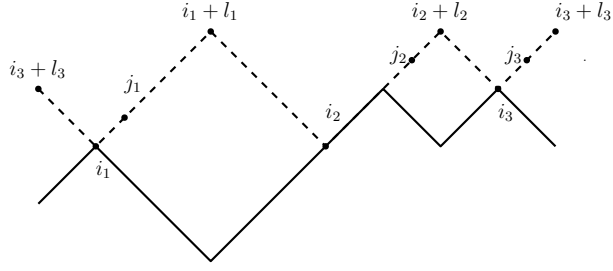}} 
\caption{The profile of  a  module  with $3$-interlacing layers, with poset $1^3\mid 2$, and with all boxes squares.} \label{fig81}
\end{center}
\end{figure}
We would like to construct an indecomposable rank 2 module 
$\MM(I,J)$ with $L_J$ as submodule and $L_I$ as quotient. 

As before, we define $x_{i}=\begin{pmatrix} t& b_{i} \\ 0 & 1 \end{pmatrix}$ and 
$y_{i}=\begin{pmatrix} 1& -b_{i} \\ 0 & t \end{pmatrix}$, if $i\in I\setminus J$,  and 
$x_{i}=\begin{pmatrix} 1& b_{i} \\ 0 & t \end{pmatrix}$ and  $y_{i}=\begin{pmatrix} t& -b_{i} \\ 0 & 1 \end{pmatrix}$ if 
$i\in J\setminus I$. 
For any other  $i$ we define $x_i$ to be the identity matrix and $y_i$ to be $t\cdot \id$ if $i\in I\cap J$, and 
$x_i=t\cdot \id$, $y_i=\id$ if $i\in I^c\cap J^c$. 
This gives an element of ${\rm CM}(B_{k,n})$, Proposition~\ref{prop:constr-works}, as we will explain now. 

In order to have a representation for $B_{k,n}$, our matrices have to satisfy the relation $x^k=y^{n-k}$. 
After multiplication by $x^{n-k}$ from the left, we conclude that this is equivalent to 
$x^n=t^{n-k}\cdot\id$ because $x_iy_i=t$. 
When we compute such a product $x_nx_{n-1}\cdots x_1$ we get the matrix 
$\begin{pmatrix} t^{n-k}& z \\ 0 & t^{n-k} \end{pmatrix}$, where $z$ is a linear combination of the coefficients $b_i$ 
{over the centre $Z$}. This linear combination must be zero if we want to have a $B_{k,n}$-module structure. 

\begin{rem}\label{parremoval}
{Note that for all $i\in(I\cap J)\cup (I^c\cap J^c)$, $x_i$ is equal to $\id$ or $t\id$, so in the product 
$x^n$, any such $x_i$ does not contribute to $z$, or more precisely, it cancels out in $x^n=t^{n-k}\cdot\id$. 
Therefore, in finding conditions for $z=0$, we will assume 
$(I\cap J)\cup (I^c\cap J^c)=\emptyset$, i.e.~that the rims of $L_I$ and of $L_J$ have no parallel segments and that 
$(k,n)$ are modified accordingly to $(k',n')$ for some $k'\le k$, $n'\le n$.  
This implies in particular that the boxes are symmetric around horizontal axes and that $n'=2k'$.}  
\end{rem}

We consider the product of all $x_i$ appearing in a single box. Since we have removed all common elements of 
$I$ and $J$ and of $I^c$ and $J^c$, the boxes are separated by the three points where the rims meet. We call them the three {\em branching points of $I\mid J$}. 
In other words, let $\{i_1,i_2,i_3\}\subset I$ be the positions where the arrow $x_{i_m}$ in the rim of $L_I$ 
ends at the position/height (in the lattice diagram) where the arrow $y_{i_m}$ starts in the rim of $L_J$, $m=1,2,3$. 
And let $\{j_1,j_2,j_3\}\subset J$ be defined as $j_m=i_m+1$. We call the $i_m$'s the branching points for $I$ and 
the $j_m$'s the branching points for $J$. (We will give the definition of branching points in the general setting later, 
cf.\ Definition~\ref{def:branching}.)

The profile $I \mid J$ of a rank 2 module with $6$-interlacing layers, with poset $1^3\mid 2$, and with $I\cap J=I^c\cap J^c=\emptyset$ 
is given in {Figure~\ref{fig:3boxes-i-j}}, for $(k,n)=(8,16)$. 
\begin{figure}[H]
\begin{center}
\includegraphics[width = 10cm]{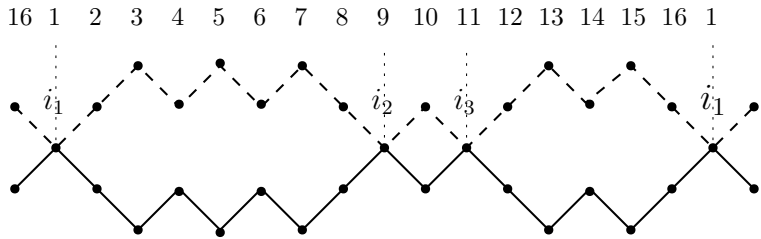} 
\caption{The profile of a module with $3$ branching points and $3$ junctions.} \label{fig:3boxes-i-j}
\end{center}
\end{figure}

Note that all the boxes now fit between consecutive branching points $i_m,i_{m+1}\in I$, but that there might be some
points, 
along the boundary of a box where the rims deviate from forming a square: considering the rim of $L_I$, 
these points are precisely the valleys of $I$ (and by symmetry, the peaks of $J$), i.e.\  points $i\in I$ (and thus $i\notin J$) such that $i+1\notin I$ (and thus $i+1\in J$). We call such a point a {\em junction of  $I$} (of $J$, by the symmetry). 
In Figure~\ref{fig:3boxes-i-j}, the first box has two junctions at 4 and at 6 in $I$. By definition, branching points 
are not junctions. 

For the remainder of this section we assume that $I\mid J$ is a profile whose boxes are squares, i.e., that there are no junctions.

\begin{figure}[H]
\begin{center}
{\includegraphics[width = 6cm]{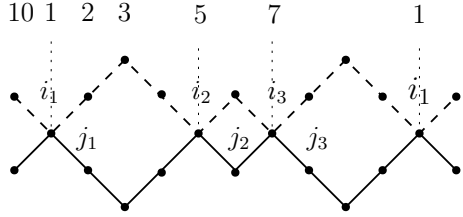}} 
\caption{A profile with squared boxes and no parallel lines.} 
\label{fig:3boxes-no-junctions}
\end{center}
\end{figure}

Consider the first box in the reduced setting (no parallel segments, only squared boxes), i.e.~the squared box 
with starting/ending points $i_1,i_2$. 
The size of the set $I\cap (i_1,i_2]$ is the same as the size of the set $J\cap (i_1,i_2]$. We 
call this number the {\em size} of the box,  and denote it by $l_1:=\frac{1}{2}|\{i_1+1,\dots, i_2\}|$, $l_2$ and $l_3$ are defined accordingly. Then the product of the matrices $x_{i_2}x_{i_2-1}\cdots x_{i_1+l_1+1}$ (see Figure \ref{fig81}) on the second half of the first box is  
$\begin{pmatrix} t^{l_1}&  b_{i_2}+b_{i_2-1}t+\dots + b_{i_1+l_1+1}t^{l_1-1} \\ 0 & 1 \end{pmatrix}.$  Note that the peaks of $I$ are $i_1+l_1,$ $i_2+l_2$, and $i_3+l_3$.

In the linear combination $b_{i_2}+b_{i_2-1}t+\dots + b_{i_1+l_1+1}t^{l_1-1}$ there is only one term that is potentially 
not divisible by $t$, the one corresponding to the branching point $i_2$. Denote the sum 
$b_{i_2}+b_{i_2-1}t+\dots + b_{i_1+l_1+1}t^{l_1-1}$ by $B_{i_2}$. 
On the first half of the box, the product of matrices $x_{i_1+l_1}x_{i_1+l_1-1}\cdots x_{i_1+1}$ is 
$\begin{pmatrix} 1&  b_{j_1}+b_{i_1+1}t+\dots + b_{i_1+l_1}t^{l_1-1} \\ 0 & t^{l_1} \end{pmatrix}.$ 
Note that $b_{j_1}$ is the only term potentially not divisible by $t$ in the sum 
$b_{j_1}+b_{j_1+1}t+\dots + b_{i_1+l_1}t^{l_1-1}$, which we denote by $B_{j_1}$. 
The product $x_{i_2}x_{i_2-1}\cdots x_{j_1}$ is thus equal to 
$\begin{pmatrix} t^{l_1}&  t^{l_1}(B_{j_1}+B_{i_2}) \\ 0 & t^{l_1} \end{pmatrix}.$  
Similarly, for the remaining two boxes we get that the corresponding products are (with the obvious analogous notation): 
$\begin{pmatrix} t^{l_2}&  t^{l_2}(B_{j_2}+B_{i_3}) \\ 0 & t^{l_2} \end{pmatrix}$  and 
$\begin{pmatrix} t^{l_3}& t^{l_3} (B_{j_3}+B_{i_1}) \\ 0 & t^{l_3} \end{pmatrix}.$ Therefore, the whole product 
$x^n$ is $\begin{pmatrix} t^{n-k}& t^{n-k} (B_{j_3}+B_{i_1}+B_{j_2}+B_{i_3}+B_{j_1}+B_{i_2} )\\ 0 & t^{n-k} \end{pmatrix}.$ 
The following proposition is now obvious. 

\begin{prop}\label{prop:constr-works}
If $B_{i_1}+ B_{i_2}+B_{i_3}+B_{j_1}+B_{j_2}+ B_{j_3} =0$, then $\MM(I,J)$ is a rank $2$ Cohen-Macaulay module.  
\end{prop}

From now on we assume that $B_{i_1}+ B_{i_2}+B_{i_3}+B_{j_1}+B_{j_2}+ B_{j_3} =0$.

 We first determine necessary and sufficient conditions for the module $\MM(I,J)$ to be isomorphic to the direct sum $L_I\oplus L_J$. Assume first that $\mathbb M(I,J)$ is decomposable and that $L_J$ is a direct summand of $\mathbb M(I,J)$. Then there exists a retraction $\mu=(\mu_i)_{i=1}^n$ such that $\mu_i\circ \theta_i=\id$, where $(\theta_i)_{i=1}^n$ is the natural injection of $L_J$ into $\mathbb M(I,J)$.  Using the same basis as before, we can assume that $\mu_i=[1\,\, \alpha_i ]$ for some 
$\alpha_i\in \CC[[t]]$. From the commutativity relations we have $\id\circ \mu_i=\mu_{i+1}\circ x_{i+1}$ for $i+1\in J\setminus I$,  and  $t\cdot \id\circ \mu_i=\mu_{i+1}\circ x_{i+1}$ for $i+1\in I\setminus J$. It follows that 
$\alpha_i=b_{i+1}+t\alpha_{i+1}$ for $i+1\in J\setminus I$, and $t\alpha_i=b_{i+1}+\alpha_{i+1}$ for $i+1\in I\setminus J$. From this we have 
\begin{align*}
t^{l_3} \alpha_{i_3+l_3}-t^{l_1}\alpha_{i_1+l_1}&=B_{i_1}+B_{j_1},\\
t^{l_2} \alpha_{i_2+l_2}-t^{l_3}\alpha_{i_3+l_3}&=B_{i_3}+B_{j_3},\\
t^{l_1} \alpha_{i_1+l_1}-t^{l_2}\alpha_{i_2+l_2}&=B_{i_2}+B_{j_2}.
\end{align*}
It follows that $t^{\min\{{l_{g-1}, l_g}\}}\mid B_{i_g}+B_{j_g}$, for $g=1,2,3$. In the opposite direction, if $t^{\min\{{l_{g-1}, l_g}\}}\mid B_{i_g}+B_{j_g}$, for $g=1,2,3$, then we can easily find the coefficients $\alpha_{i_g+l_g}$, $g=1,2,3,$ so that the above equalities hold.  Assume, without loss of generality, that $l_3\geq l_2\geq l_1.$ Then by setting $\alpha_{i_3+l_3}=0$ we obtain $\alpha_{i_1+l_1}=-t^{-l_1}(B_{i_1}+B_{j_1})$ and $\alpha_{i_2+l_2}=t^{-l_2}(B_{i_3}+B_{j_3})$. The coefficient $\alpha_{i+1}$, where $i+1$ is not a peak of $I$, is determined directly from $\alpha_i=b_{i+1}+t\alpha_{i+1}$, when $i+1\in J\setminus I$, or $t\alpha_i=b_{i+1}+\alpha_{i+1}$, when $i+1\in I\setminus J$. This defines a retraction $\mu$, and thus proves the following proposition. 

\begin{prop}\label{sum}The module $\MM(I,J)$ is isomorphic to $L_I\oplus L_J$ if and only if $t^{\min\{{l_{g-1}, l_g}\}}\mid B_{i_g}+B_{j_g}$, for $g=1,2,3$.
 \end{prop}
 
 We give necessary and sufficient conditions for the module $\MM(I,J)$ to be indecomposable.
 \begin{theorem}\label{thm:indec}
The  module $\MM(I,J)$ is indecomposable if and only if  $t^{\min\{l_3,l_1\}}\nmid B_{i_1}+B_{j_1}$, $t^{\min\{l_1,l_2\}}\nmid B_{i_2}+B_{j_2}$, and  
$t^{\min\{l_2,l_3\}}\nmid B_{i_3}+B_{j_3}$.  
\end{theorem}
\begin{proof} Let $\varphi=(\varphi_{i_1})_{i=1}^n$ be an idempotent endomorphism of $\MM(I,J)$ and let 
$\varphi_{i_3+l_3}= \begin{pmatrix} a& b \\ c & d \end{pmatrix}$ (see Figure~\ref{fig81}). 
From  $x_{i_1+l_1}x_{i_1+l_1-1}\cdots x_{j_1}\cdot x_{i_1}\cdots x_{i_3+l_3+1}\varphi_{i_3+l_3}=\varphi_{i_1+l_1} x_{i_1+l_1}x_{i_1+l_1-1}\cdots x_{j_1}\cdot x_{i_1}\cdots x_{i_3+l_3+1}$ follows that 
$$\varphi_{i_1+l_1}= \begin{pmatrix} a+t^{-l_3}(B_{i_1}+B_{j_1})c& t^{-l_1}[ b t^{l_3}+(d-a)(B_{i_1}+B_{j_1})-t^{-l_3}(B_{i_1}+B_{j_1})^2c] \\ t^{l_1-l_3}c & d-t^{-l_3}(B_{i_1}+B_{j_1})c \end{pmatrix}, $$ 
where $t^{l_3-l_1}\mid c$, $t^{l_1}\mid (d-a)(B_{i_1}+B_{j_1})-t^{-l_3}(B_{i_1}+B_{j_1})^2c$, and $t^{l_3}\mid (B_{i_1}+B_{j_1})c$. Here, we assume, without loss of generality, that   $l_1\leq l_2\leq l_3$.  

From  $x_{i_3+l_3}\cdots x_{j_3}\cdot x_{i_3}\cdots x_{i_2+l_2+1}\varphi_{i_2+l_2}=\varphi_{i_3+l_3} x_{i_3+l_3}\cdots x_{j_3}\cdot x_{i_3}\cdots x_{i_2+l_2+1}$ follows that 
$$\varphi_{i_2+l_2}= \begin{pmatrix} a-t^{-l_3}(B_{i_3}+B_{j_3})c& t^{-l_2}[b t^{l_3}-(d-a)(B_{i_3}+B_{j_3})-t^{-l_3}(B_{i_3}+B_{j_3})^2c] \\ t^{l_2-l_3}c & d+t^{-l_3}(B_{i_3}+B_{j_3})c \end{pmatrix}, $$ 
where $t^{l_3-l_2}\mid c$, $t^{l_2}\mid (d-a)(B_{i_3}+B_{j_3})+t^{-l_3}(B_{i_3}+B_{j_3})^2c$, and $t^{l_3}\mid (B_{i_3}+B_{j_3})c$. 

Assume that $t^{\min\{l_3,l_1\}}\nmid B_{i_1}+B_{j_1}$, $t^{\min\{l_1,l_2\}}\nmid B_{i_2}+B_{j_2}$, 
and $t^{\min\{l_2,l_3\}}\nmid B_{i_3}+B_{j_3}$. Then for $g=1,2,3,$ there exists $s_g$
 such that $0\leq s_g<\min\{l_{g-1},l_g\}$, $t^{s_g}\mid B_{i_g}+B_{j_g}$, and $t^{s_g+1}\nmid B_{i_g}+B_{j_g}$. Note that since $B_{i_1}+ B_{i_2}+B_{i_3}+B_{j_1}+B_{j_2}+ B_{j_3} =0$ it is not possible that one of the $s_g$'s is strictly less than the other two, i.e.\ the smallest two have to be equal. For the coefficients of $\varphi_{i_3+j_3}$ it holds  that
 \begin{align*}
 t^{l_1-s_1}&\mid (d-a)t^{-s_1}(B_{i_1}+B_{j_1})-t^{-s_1}(B_{i_1}+B_{j_1})t^{-l_3}(B_{i_1}+B_{j_1})c,\\
 t^{l_2-s_3}&\mid  (d-a)t^{-s_3}(B_{i_3}+B_{j_3})+t^{-s_3}(B_{i_3}+B_{j_3})t^{-l_3}(B_{i_3}+B_{j_3})c.
 \end{align*}

 It follows  that $t^{l_1-s_1}\mid (d-a)-t^{-l_3}(B_{i_1}+B_{j_1})c$ and $t^{l_2-s_3}\mid (d-a)+t^{-l_3}(B_{i_3}+B_{j_3})c$.  
 Then $t^{\min\{l_1-s_1, l_2-s_3\}}\mid t^{-l_3}(B_{i_2}+B_{j_2})c$ and subsequently $t^{\min\{l_1-s_1,l_2-s_3\}}\mid t^{-l_3}t^{s_2}c$.  If $\min\{l_1-s_1,l_2-s_3\}=l_2-s_3$, then $t^{l_2-s_3}\mid t^{-l_3}t^{s_2}t^{-s_3}(B_{i_3}+B_{j_3})c$ and $t^{l_2}\mid t^{-l_3}t^{s_2}(B_{i_3}+B_{j_3})c$. Since $l_2>s_2$, we have $t\mid t^{-l_3}(B_{i_3}+B_{j_3})c$ and $t\mid d-a$. Assume now that $\min\{l_1-s_1,l_2-s_3\}=l_1-s_1$.  From $t^{l_1-s_1}\mid t^{-l_3}t^{s_2}t^{-s_1}(B_{i_1}+B_{j_1})c$ we conclude directly that $t^{l_1-s_2}\mid t^{-l_3}(B_{i_1}+B_{j_1})c$. Since $l_1>s_2$, we have  that $t\mid t^{-l_3}(B_{i_1}+B_{j_1})c$ and $t\mid d-a$.

 From  $t\mid d-a$ follows  that $a=d$ and $b=c=0$, giving us the trivial idempotents. Hence, the module $\MM(I,J)$ is indecomposable.   
 
 Assume that at least one of the conditions $t^{\min\{l_3,l_1\}}\nmid B_{i_1}+B_{j_1}$, $t^{\min\{l_1,l_2\}}\nmid B_{i_2}+B_{j_2}$, and 
$t^{\min\{l_2,l_3\}}\nmid B_{i_3}+B_{j_3}$ does not hold.
 If {$t^{\min\{l_2,l_3\}}\mid  B_{i_3}+B_{j_3}$}, $t^{\min\{l_3,l_1\}}\mid  B_{i_1}+B_{j_1}$, and  $t^{\min\{l_1,l_2\}}\mid  B_{i_2}+B_{j_2}$, then, by Proposition~\ref{sum}, $\MM(I,J)\cong L_I\oplus L_J.$

If {$t^{\min\{l_2,l_3\}}\mid  B_{i_3}+B_{j_3}$}, and $t^{\min\{l_3,l_1\}}\nmid  B_{i_1}+B_{j_1}$ or $t^{\min\{l_1,l_2\}}\nmid  B_{i_2}+B_{j_2}$, then we repeat the same procedure as in the previous section to construct a non-trivial idempotent. Assume that $t^{\min\{l_3,l_1\}}\nmid  B_{i_1}+B_{j_1}$ (analogous arguments are used if we were to assume that $t^{\min\{l_1,l_2\}}\nmid  B_{i_2}+B_{j_2}$). The only divisibility conditions that the elements of the matrix $\varphi_{i_3+l_3}= \begin{pmatrix} a& b \\ c & d \end{pmatrix}$ have to fulfill are $t^{l_3-s_1}\mid c$ and $t^{l_1}\mid (d-a)(B_{i_1}+B_{j_1})-t^{-l_3}(B_{i_1}+B_{j_1})^2c$. The latter condition is equivalent to the condition $t^{l_1-s_1}\mid (d-a)-t^{-l_3}(B_{i_1}+B_{j_1})c$.

Recall that if $a=d$, then we only obtain the trivial idempotents because $t\mid c$. So it must be $a+d=1$.  If we set $a=1$, $d=0$, $c=-t^{l_3-s_1}(t^{-s_1}(B_{i_1}+B_{j_1}))^{-1}$, and $b=0$, we  get the idempotent  $$\varphi_{i_3+l_3}=\begin{pmatrix}
1& 0\\
-t^{l_3-s_1}(t^{-s_1}(B_{i_1}+B_{j_1}))^{-1}&0
\end{pmatrix}.$$ 
 Its orthogonal complement is the idempotent
 $$
 \begin{pmatrix}
0& 0\\
t^{l_3-s_1}(t^{-s_1}(B_{i_1}+B_{j_1}))^{-1}&1
\end{pmatrix}.$$  From $x_i\varphi_{i-1}=\varphi_ix_i$, we easily determine idempotents $\varphi_i$, for all $i$. 
Since these  are non-trivial idempotents, it follows that the module $\MM(I,J)$ is decomposable, what we needed to prove. 

The case when $t^{\min\{l_3,l_1\}}\mid  B_{i_1}+B_{j_1}$, and {$t^{\min\{l_2,l_3\}}\nmid  B_{i_3}+B_{j_3}$}  or $t^{\min\{l_1,l_2\}}\nmid  B_{i_2}+B_{j_2}$, and the case when $t^{\min\{l_1,l_2\}}\mid  B_{i_2}+B_{j_2}$, and $t^{\min\{l_3,l_1\}}\nmid  B_{i_1}+B_{j_1}$  or {$t^{\min\{l_2,l_3\}}\nmid  B_{i_3}+B_{j_3}$}   are treated similarly, so we omit the details.   
\end{proof}
 
 \begin{rem}
All arguments from the previous theorem hold more generally, for example, when the boxes are rectangular. By Remark \ref{parremoval}, the matrices $x_i$ are all scalar for the parts that are parallel, so they can be ignored in all computations. So the arguments hold  for $\MM(I,J)$ where the profiles can be reduced to a profile satisfying the conditions of the theorem by removing indices in $(I\cap J)\cup (I^c\cap J^c)$ from the $k$-subsets.
\end{rem}
 
\begin{rem}  As in Proposition~\ref{sumica}, it is possible to give an explicit combinatorial description of the summands of $\MM(I,J)$ when this module is decomposable. We do not give this combinatorial description here, but we note that the coefficients $s_1,s_2,$ and $s_3$ from the proof of the previous theorem play a crucial role in determining the summands. For example, if {$t^{\min\{l_2,l_3\}}\mid  B_{i_3}+B_{j_3}$}, $t^{\min\{l_3,l_1\}}\nmid  B_{i_1}+B_{j_1}$, $t^{\min\{l_1,l_2\}}\nmid  B_{i_2}+B_{j_2}$, and $s_1=s_2=0$,   then $\MM(I,J)$ is isomorphic to $L_X \oplus L_Y$ where $X=(J\cup (I\cap (i_3,i_1]))\setminus  (J\cap (i_3,i_1]))$ and $Y=(I\cup (J\cap (i_3,i_1])))\setminus ( I\cap (i_3,i_1]).$ The result of Proposition~\ref{sumica} is a special case of this result. 
\end{rem}

\begin{corollary} In the case $(4,8)$, if $I=\{2,4,7,8\}$ and $J=\{1,3,5,6\}$, then $\MM(I,J)$ is indecomposable if and only if $t\nmid b_2+b_3$, $t\nmid b_4+b_5$, and $t\nmid b_8+b_1$.  Furthermore, if   $t\nmid b_2+b_3$, $t\mid b_4+b_5$, $t\nmid b_8+b_1$, then  $\MM(I,J)$ is isomorphic to $L_{\{2,3,5,6\}}\oplus L_{\{1,4,7,8\}}$. If   $t\mid b_2+b_3$, $t\nmid b_4+b_5$, $t\nmid b_8+b_1$, then  $\MM(I,J)$ is isomorphic to  $L_{\{2,4,5,6\}}\oplus L_{\{1,3,7,8\}}$.  If  $t\nmid b_2+b_3$, $t\nmid b_4+b_5$, $t\mid b_8+b_1$, then  $\MM(I,J)$ is isomorphic to  $L_{\{2,3,7,8\}}\oplus L_{\{1,4,5,6\}}$.  If  $t\mid b_2+b_3$, $t\mid b_4+b_5$, $t\mid b_8+b_1$, then  $\MM(I,J)$ is isomorphic to  $L_I\oplus L_J$.
\end{corollary}

We now consider the question of uniqueness of the rank 2 modules we constructed. In case we have a decomposable rank 2 module, the module is completely determined by the divisibility conditions its coefficients $b_i$, $i=1,\ldots,n$, satisfy. Thus, two sets of coefficients $b_i$ define isomorphic decomposable modules if and only if they satisfy the same divisibility conditions. What remains to study is the case when we have an indecomposable module. The next theorem tells us that in this case too, there is a unique rank 2 indecomposable module with $L_{\{1,3,5,6\}}$ as a submodule and $L_{\{2,4,7,8\}}$ as a quotient.

Let $(b_i)$ and $(c_i)$ be different $n$-tuples corresponding to indecomposable modules $\MM_1$ and $\MM_2$ with filtration $L_I\mid L_J$ from Theorem~\ref{thm:indec}, {satisfying the indecomposability conditions of Theorem~\ref{thm:indec}}.  Thus, assume that $t^{\min\{l_{g-1},l_{g}\}}\nmid B_{i_g}+B_{j_g}$, $t^{\min\{l_{g-1},l_{g}\}}\nmid C_{i_g}+C_{j_g}$, for $g=1,2,3.$ Then for each $g$, there exists $s^b_g$ such that  $0\leq s^b_g<\min\{l_{g-1},l_g\}$,  $t^{s^b_g}\mid B_{i_g}+B_{j_g}$, and $t^{s^b_g+1}\nmid B_{i_g}+B_{j_g}$, and there exists $s^c_g$ such that  $0\leq s^c_g<\min\{l_{g-1},l_g\}$,  $t^{s^c_g}\mid C_{i_g}+C_{j_g}$, and $t^{s^c_g+1}\nmid C_{i_g}+C_{j_g}$.
Recall from the proof of Theorem~\ref{thm:indec} that since $B_{i_1}+ B_{i_2}+B_{i_3}+B_{j_1}+B_{j_2}+ B_{j_3} =0$ it is not possible that one of the $s^b_g$'s is strictly less than the other two, i.e.\ the smallest two have to be equal. The same holds for $s^c_g$'s.

\begin{theorem} \label{thm:unique}
Let $(b_i)$ and $(c_i)$ be different $n$-tuples corresponding to indecomposable modules $\MM_1$ and $\MM_2$ with filtration 
$L_I\mid L_J$ {satisfying the indecomposability conditions of Theorem~\ref{thm:indec}}. 
Then the modules $\MM_1$ and $\MM_2$ are isomorphic if and only if $s_g^c=s_g^b$ for $g=1,2,3$. 
\end{theorem}

\begin{proof}
Let us assume that there is an isomorphism $\varphi=(\varphi_i)_{i=1}^n$  between $\MM_1$ and $\MM_2$. If $\varphi_{i_3+l_3}=\begin{pmatrix}\alpha& \beta \\ \gamma & \delta  \end{pmatrix}$, then from $x_{i_1+l_1}x_{i_1+l_1-1}\cdots x_{j_1}\cdot x_{i_1}\cdots x_{i_3+l_3+1}\varphi_{i_3+l_3}=\varphi_{i_1+l_1} x_{i_1+l_1}x_{i_1+l_1-1}\cdots x_{j_1}\cdot x_{i_1}\cdots x_{i_3+l_3+1}$ follows that $\varphi_{i_1+l_1}$ is equal to 
$$
\begin{pmatrix} \alpha+t^{-l_3}(C_{i_1}+C_{j_1})\gamma & t^{-l_1}[t^{l_3}\beta-\alpha (B_{i_1}+B_{j_1}) + (C_{i_1}+C_{j_1})\delta-(B_{i_1}+B_{j_1})t^{-l_3}(C_{i_1}+C_{j_1})\gamma]\\ t^{l_1-l_3}\gamma& \delta-t^{-l_3}(B_{i_1}+B_{j_1})\gamma \end{pmatrix},$$
and $\varphi_{i_2+l_2}$ is equal to 
$$\begin{pmatrix} \alpha-t^{-l_3}(C_{i_3}+C_{j_3})\gamma & t^{-l_2}[t^{l_3}\beta+\alpha (B_{i_3}+B_{j_3}) -(C_{i_3}+C_{j_3})\delta-(B_{i_3}+B_{j_3})t^{-l_3}(C_{i_3}+C_{j_3})\gamma]\\ t^{l_2-l_3}\gamma & \delta+t^{-l_3}(B_{i_3}+B_{j_3})\gamma \end{pmatrix},$$
where $t^{l_3-l_2}\mid \gamma$, $t^{l_3-l_1}\mid \gamma$, $t^{l_3}\mid (B_{i_1}+B_{j_1})\gamma$, $t^{l_3} \mid (C_{i_1}+C_{j_1})\gamma$, $t^{l_1}\mid -\alpha (B_{i_1}+B_{j_1}) + (C_{i_1}+C_{j_1})\delta-(B_{i_1}+B_{j_1})t^{-l_3}(C_{i_1}+C_{j_1})\gamma$, and $t^{l_2}\mid \alpha (B_{i_3}+B_{j_3}) -(C_{i_3}+C_{j_3})\delta-(B_{i_3}+B_{j_3})t^{-l_3}(C_{i_3}+C_{j_3})\gamma$. Here, without loss of generality we assume that $l_3$ is the largest amongst $l_1,l_2,l_3$.  The coefficients $\alpha$ and $\delta$ are invertible because $t\mid \gamma$ and $\varphi_{i_3+l_3}$ is invertible.

If $s_1^b>s^c_1$, then $t^{-s_1^c}(B_{i_1}+B_{j_1})$ is divisible by $t$. From $t^{l_1}\mid -\alpha (B_{i_1}+B_{j_1}) + (C_{i_1}+C_{j_1})\delta-(B_{i_1}+B_{j_1})t^{-l_3}(C_{i_1}+C_{j_1})\gamma$ it follows that $$t^{l_1-s_1^c}\mid -\alpha t^{-s_1^c}(B_{i_1}+B_{j_1}) + t^{-s_1^c}(C_{i_1}+C_{j_1})\delta-t^{-s_1^c}(B_{i_1}+B_{j_1})t^{-l_3}(C_{i_1}+C_{j_1})\gamma.$$ This implies that $t\mid  t^{-s_1^c}(C_{i_1}+C_{j_1})\delta$, but neither $\delta$ nor  $t^{-s_1^c}(C_{i_1}+C_{j_1})$ is divisible by $t$. This is a contradiction. The proof is analogous if $s_2^b>s^c_2$ or $s_3^b>s^c_3$. Thus if  $s_g^c\neq s_g^b$ for at least one $g$, the modules $\MM_1$ and $\MM_2$ are not isomorphic.

Assume that $s_g^c=s_g^b$ for $g=1,2,3$. From $t^{l_1}\mid -\alpha (B_{i_1}+B_{j_1}) + (C_{i_1}+C_{j_1})\delta-(B_{i_1}+B_{j_1})t^{-l_3}(C_{i_1}+C_{j_1})\gamma$ and $t^{l_2}\mid \alpha (B_{i_3}+B_{j_3}) -(C_{i_3}+C_{j_3})\delta-(B_{i_3}+B_{j_3})t^{-l_3}(C_{i_3}+C_{j_3})\gamma$ we have 
$$t^{\min\{l_1,l_2\}}\mid \alpha (B_{i_1}+B_{j_1}) (C_{i_2}+C_{j_2}) (B_{i_3}+B_{j_3}) - \delta (C_{i_1}+C_{j_1})  (B_{i_2}+B_{j_2}) (C_{i_3}+C_{j_3}).$$ 

If we set $\alpha (B_{i_1}+B_{j_1}) (C_{i_2}+C_{j_2}) (B_{i_3}+B_{j_3}) - \delta (C_{i_1}+C_{j_1})  (B_{i_2}+B_{j_2}) (C_{i_3}+C_{j_3})=0$, then by defining  $\delta=1$, we get  $$\alpha=\displaystyle\frac{ [t^{-s_1^c}(C_{i_1}+C_{j_1})][t^{-s_2^b}(B_{i_2}+B_{j_2})][t^{-s_3^c}(C_{i_3}+C_{j_3})]}{[t^{-s_1^b}(B_{i_1}+B_{j_1})][t^{-s_2^c}(C_{i_2}+C_{j_2})][t^{-s_3^b}(B_{i_3}+B_{j_3})]}. $$

By setting  $-\alpha (B_{i_1}+B_{j_1}) + (C_{i_1}+C_{j_1})\delta-(B_{i_1}+B_{j_1})t^{-l_3}(C_{i_1}+C_{j_1})\gamma=0$, we obtain $$\gamma=t^{l_3} \displaystyle\frac{[t^{-s_2^c}(C_{i_2}+C_{j_2})][t^{-s_3^b}(B_{i_3}+B_{j_3})]-[t^{-s_2^b}(B_{i_2}+B_{j_2})][t^{-s_3^c}(C_{i_3}+C_{j_3})]}{[t^{-s_1^b}(B_{i_1}+B_{j_1})][t^{-s_2^c}(C_{i_2}+C_{j_2})][t^{-s_3^b}(B_{i_3}+B_{j_3})]}.$$

The coefficients $\alpha, \beta$, and $\gamma$ of $\varphi_{i_3+l_3}$ satisfy the necessary divisibility conditions, it is left to set $\beta=0$ to obtain the isomorphism {\renewcommand\arraystretch{2.5}
$$\varphi_{i_3+l_3}=\begin{pmatrix} \displaystyle\frac{ [t^{-s_1^c}(C_{i_1}+C_{j_1})][t^{-s_2^b}(B_{i_2}+B_{j_2})][t^{-s_3^c}(C_{i_3}+C_{j_3})]}{[t^{-s_1^b}(B_{i_1}+B_{j_1})][t^{-s_2^c}(C_{i_2}+C_{j_2})][t^{-s_3^b}(B_{i_3}+B_{j_3})]}& \,\,\,\,\,\,\,\, 0 \\  t^{l_3} \displaystyle\frac{[t^{-s_2^c}(C_{i_2}+C_{j_2})][t^{-s_3^b}(B_{i_3}+B_{j_3})]-[t^{-s_2^b}(B_{i_2}+B_{j_2})][t^{-s_3^c}(C_{i_3}+C_{j_3})]}{[t^{-s_1^b}(B_{i_1}+B_{j_1})][t^{-s_2^c}(C_{i_2}+C_{j_2})][t^{-s_3^b}(B_{i_3}+B_{j_3})]} & \,\,\,\,\,\,\,\,1  \end{pmatrix}.$$
Once we know $\varphi_{i_3+l_3}$, $\varphi_{i_2+l_2}$, and $\varphi_{i_1+l_1}$, we easily compute $\varphi_i$, for all $i$, again by using relations $x_i\varphi_{i-1}=\varphi_ix_i.$}
\end{proof}

\begin{corollary} \label{talmostcor}
In the tame case $(4,8)$, up to isomorphism, there is a unique rank $2$ indecomposable module with filtration layers $2478|1356$.
\end{corollary}
\begin{proof} In this case, $s_g=0$, for $g=1,2,3$, and the statement follows from the previous theorem.
\end{proof}

 \begin{ex} Observe the case  $(5,10)$. Let $I=\{1,2,5,6,8\}$ and  $J=\{3,4,7,9,10\}$. For this profile, there are three boxes of sizes $l_1=1,$ $l_2=2$, and $l_3=2$. 
 
 Define $b_1=0$, $b_2=t$, $b_3=b_4=b_5=0$, $b_6=1$, $b_7=-t$, $b_8=-1$, and $b_9=b_{10}=0$. It holds that $\sum_{i=1}^{10}b_i=0$, $t\nmid b_7+b_6+tb_5$, $t\nmid b_8+b_9+tb_{10}$, $t^2\nmid  b_2+tb_1+b_3+tb_4$, and $t\mid b_2+tb_1+b_3+tb_4$. Thus, $s^b_1=s^b_2=0$ and $s^b_3=1$. Denote the corresponding module by $\MM_1$. This module is indecomposable by Theorem~\ref{thm:indec}. 
 
Define   $c_1=0$, $c_2=1$, $c_3=c_4=c_5=0$, $c_6=1$, $c_7=0$, $c_8=-2$, and $c_9=c_{10}=0$. It holds that $\sum_{i=1}^{10}c_i=0$, $t\nmid c_7+c_6+tc_5$, $t\nmid c_8+c_9+tc_{10}$, $t^2\nmid  c_2+tc_1+c_3+tc_4$, and $t\nmid c_2+tc_1+c_3+tc_4$. Thus, $s^c_1=s^c_2=s^c_3=0$. Denote the corresponding module by $\MM_2$. This module is indecomposable by Theorem~\ref{thm:indec}. 

By Theorem~\ref{thm:unique}, the modules $\MM_1$ and $\MM_2$ are not isomorphic.  
 
 \end{ex}

 \begin{corollary}
 Let $I\mid J$ be the filtration from Theorem~\ref{thm:indec} and assume that $l_1\leq l_2\leq l_3$. Up to isomorphism, there are $$l_1(\frac {l_1-1}{2}+l_2)$$     
  indecomposable modules with filtration $I\mid J$. 
 \end{corollary}
\begin{proof}  In order to count the non-isomorphic modules, we have to count the number of different triples $s_1,s_2,s_3$. Recall that the smallest two among $s_1,s_2$ and $s_3$ have to be equal because $B_{i_1}+ B_{i_2}+B_{i_3}+B_{j_1}+B_{j_2}+ B_{j_3} =0$. Also, recall that $s_i<\min\{s_{i-1},s_i\}$.

Assume that  $s_1=i$. If $s_2$ or $s_3$ is less than $s_1$, then $s_3=s_2$ and there are $i$ different choices for $s_3=s_2$. If $s_1$ is less than $s_2$, then $s_1$ has to be equal to $s_3$ and vice versa, if  $s_1$ is less than $s_3$, then $s_1$ has to be equal to $s_2$. In total there are $l_1-i-1$ options in the former case, and $l_2-i-1$ in the latter case. We also have to count the case when $s_1=s_2=s_3=i$. Thus, there are $$\sum_{i=0}^{l_1-1}(i+1+l_1-i-1+l_2-i-1)=\sum_{i=0}^{l_1-1}(l_1+l_2-i-1)=l_1(\frac {l_1-1}{2}+l_2).$$    
\end{proof}

The expression $l_1(\frac {l_1-1}{2}+l_2)$ is equal to 1 if and only if $l_1=l_2=1$. In all other cases it is strictly greater than 1. Thus, there is a unique indecomposable module for a given filtration only when $l_1=l_2=1$, i.e. only when two of the boxes of the profile are of size 1. We give this special type of a profile a name. We say that $I$ and $J$ are {\em almost tightly $3$-interlacing} if $I\setminus J=\{a_1\}\cup\{a_2\}\cup \{a_3, \dots, a_{3+r}\}$ and $J\setminus I=\{b_1\}\cup\{b_2\}\cup \{b_3, \dots, b_{3+r}\}$, $r\geq 0$, and $b_1<a_1<b_2<a_2<b_3< \dots <b_{3+r}<a_3< \dots < a_{3+r}.$  Combinatorially,  $I$ and $J$ are almost tightly $3$-interlacing  if the profile $I\mid J$ has three squared boxes, at least two of them of size 1. Here, we assume that the potential parallel lines of the rims have been removed in order to simplify notation. Note that, by definition, tightly $3$-interlacing layers are also almost tightly $3$-interlacing.  


  \begin{corollary}
 Let $I\mid J$ be the filtration from Theorem~\ref{thm:indec}.  There is a unique rank $2$ indecomposable module with filtration $I\mid J$ if and only if $I$ and $J$ are almost tightly $3$-interlacing. 
 \end{corollary}
 
 \begin{rem}   In the general case of layers with profile $1^3\mid 2$, i.e.\ in the case  when $I$ and $J$ form three boxes that are not necessarily squares or rectangles (meaning that there are junctions), one can construct multiple non-isomorphic indecomposable modules with the given filtration. To see this, when constructing indecomposable modules we deal with  junction points in the following way. Observe {Figure~\ref{fig:3boxes-i-j}}. For the junction points 4 and 6 of the leftmost box, we define $b_4=b_5=b_6=b_7=0$. Similarly, for the rightmost box we set $b_{14}=b_{15}=0$. In this way we practically ignore the junctions because for each junction point $i$ it holds that $x_{i+1}x_i=t\cdot {\rm id}$ so it does not interfere with our computation, and we can behave as if we were in the case when all boxes are squares. If there are at least two boxes with junction points, then as in the case when all boxes are squares it follows that there are multiple non-isomorphic indecomposable modules. In the case when there is only one box with junction points, then we use the arguments from the last section (by treating one of the junctions as if it were a branching point) to construct multiple non-isomorphic indecomposable modules. 
 \end{rem}

\begin{ex} 
In the tame case $(4,8)$, there is only one type of a profile that is almost tightly $3$-interlacing. 
Such a profile is $2478\mid 1356$ {(and all profiles obtained from this profile by adding $a$ to each element of both $4$-subsets, 
for $a=1,\dots, 7$)}. 
To construct modules with this profile we define $x_i=\begin{pmatrix} t& b_{i} \\ 0 & 1 \end{pmatrix}$, 
$y_i=\begin{pmatrix} 1& -b_{i} \\ 0 & t \end{pmatrix},$ for $i=2,4,7,8$ and 
$x_{i}=\begin{pmatrix} 1& b_{i} \\ 0 & t \end{pmatrix},$ $y_i=\begin{pmatrix} t& -b_{i} \\ 0 & 1 \end{pmatrix},$ for  $i=1,3,5,6$.   
Note that the $x_i$'s are almost the same as for a module with the profile $2468\mid 1357$ constructed in the next 
section, with only the ones at vertices 6 and 7 changing places. In order for this to be a module we assume that 
$b_1+b_2+b_3+b_4+b_5+b_8+t(b_6+b_7)=0.$ 
 Denote this module again by $\MM(I,J)$. As for the module $135\mid 246$, it is easily seen that $L_J$ is a summand of $\MM(I,J)$ if and only if $t\mid b_8+b_1$, $t\mid b_2+b_3$, and $t\mid b_4+b_5$. The module $\MM(I,J)$ is indecomposable if and only if  $t\nmid b_2+b_3$, $t\nmid b_4+b_5$, $t\nmid b_8+b_1$. If   $t\nmid b_2+b_3$, $t\mid b_4+b_5$, $t\nmid b_8+b_1$, then  $\MM(I,J)$ is isomorphic to $L_{\{2,3,5,6\}}\oplus L_{\{1,4,7,8\}}$. If   $t\mid b_2+b_3$, $t\nmid b_4+b_5$, $t\nmid b_8+b_1$, then  $\MM(I,J)$ is isomorphic to  $L_{\{2,4,5,6\}}\oplus L_{\{1,3,7,8\}}$.  If  $t\nmid b_2+b_3$, $t\nmid b_4+b_5$, $t\mid b_8+b_1$, then  $\MM(I,J)$ is isomorphic to  $L_{\{2,3,7,8\}}\oplus L_{\{1,4,5,6\}}$.

There are four different decomposable modules appearing as the middle term in a short exact sequence that has  $L_I$ (as a quotient) and $L_J$ (as a submodule) as end terms: 
\begin{align*}
0\longrightarrow L_J\longrightarrow L_{\{1, 3, 5,6\}}\oplus L_{\{2,4, 7,8\}} \longrightarrow L_I \longrightarrow 0,\\
0\longrightarrow L_J\longrightarrow L_{\{1, 4, 7,8 \}}\oplus L_{\{2,3,5, 6\}} \longrightarrow L_I \longrightarrow 0,\\
0\longrightarrow L_J\longrightarrow L_{\{1, 3, 7,8\}}\oplus L_{\{2,4,5, 6\}} \longrightarrow L_I \longrightarrow 0,\\
0\longrightarrow L_J\longrightarrow  L_{\{2,3,7,8\}}  \oplus L_{\{1, 4, 5,6 \}}\longrightarrow L_I \longrightarrow 0.
\end{align*}
The {pairs of profiles} of the four modules that appear in the middle in these short exact sequences can be 
pictured as follows (similarly as in Example~\ref{ex:3-6-through}). 

\begin{figure}[H]
\begin{center}
\subfloat[$L_{\{1, 3, 5,6\}}\oplus L_{\{2,4, 7,8\}}$]{\includegraphics[width = 6cm]{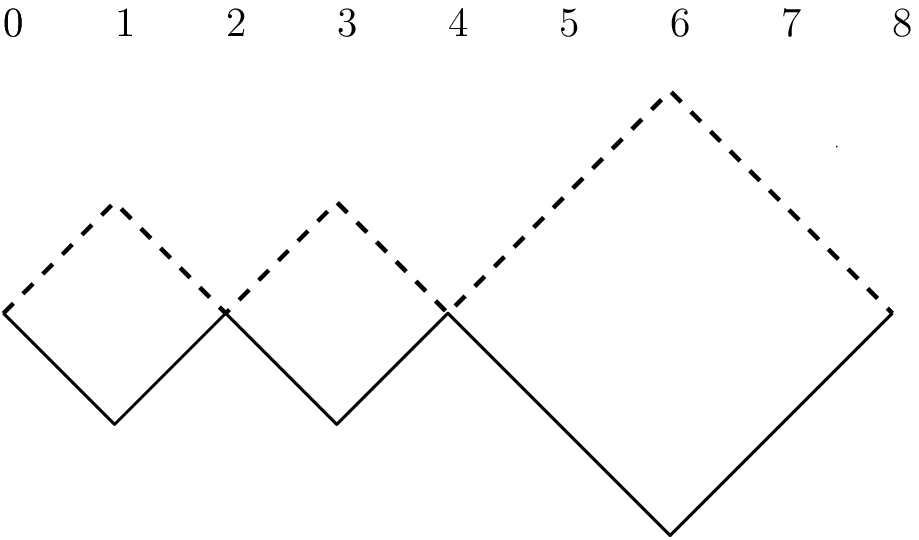}}  \quad \quad \quad \quad
\subfloat[$L_{\{1, 4, 7,8\}}\oplus L_{\{2,3,5, 6\}}$]{\includegraphics[width = 6cm ]{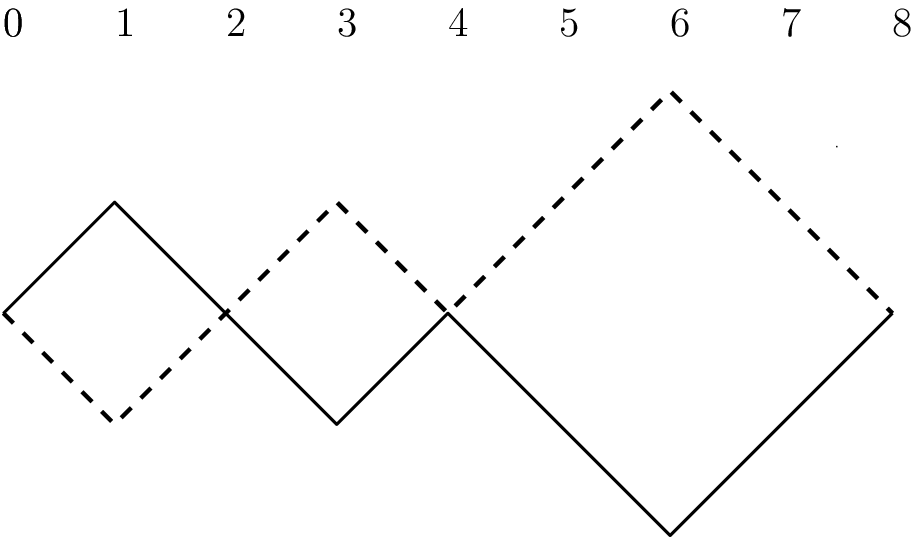}}\\
\subfloat[$L_{\{1, 3, 7,8\}}\oplus L_{\{2,4, 5, 6\}}$]{\includegraphics[width = 6cm ]{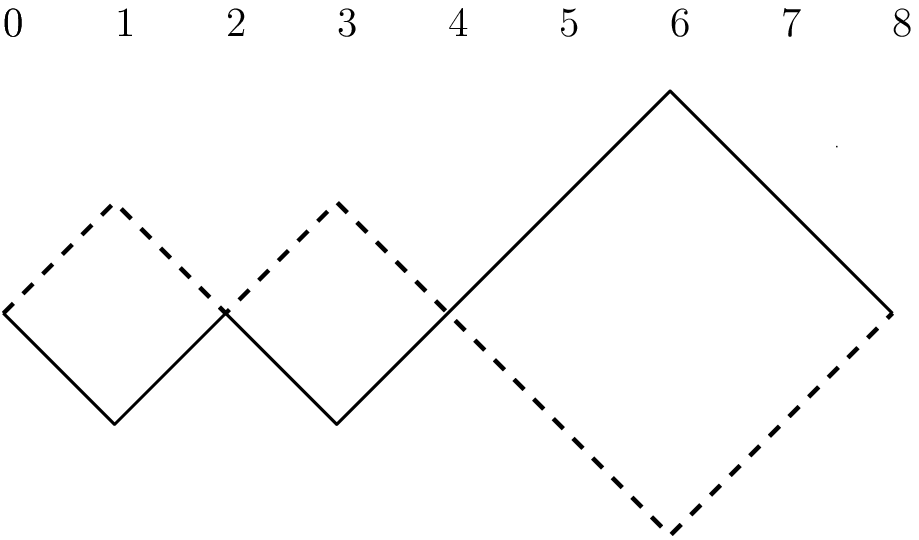}}  \quad \quad \quad \quad
\subfloat[$L_{\{2,3,7,8\}} \oplus L_{\{1, 4, 5,6\}} $]{\includegraphics[width = 6cm ]{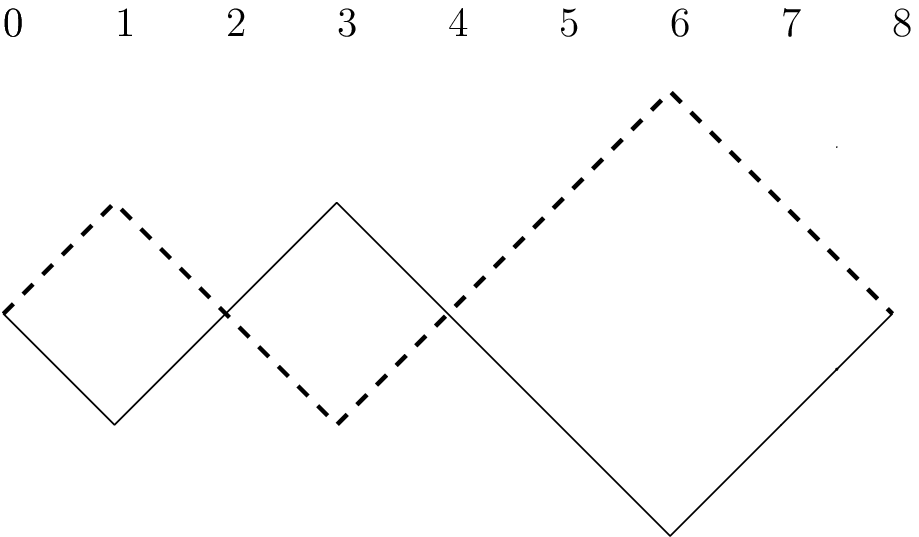}} 
\caption{{The pairs of} profiles of decomposable extensions between $ L_{\{1, 3, 5,6\}}$ and $L_{\{2,4, 7,8\}}$.}
\end{center}
\end{figure}
\end{ex}

\section{Tight $4$-interlacing}\label{sec:4-interlacing}
In the tame case $(4,8)$, there is only one type of configuration of layers with $4$-interlacing,  $1357\mid 2468$ 
(and the one obtained by adding $1$ to each element of the two $4$-subsets). We study this now. Let 
$I=\{1,3,5,7\}$ and $J=\{2,4,6,8\}.$ 
The construction is the same as for the module $135\mid 246$, we just have two more vertices. 
So assume that $x_i=\begin{pmatrix} t& b_{i} \\ 0 & 1 \end{pmatrix}$ for odd $i$ and 
$x_{i}=\begin{pmatrix} 1& b_{i} \\ 0 & t \end{pmatrix}$ for even $i$. 
From $x^k=y^{n-k}$ it follows that $\sum_1^8b_i=0$. We denote the constructed module by $\MM$ and  study the structure of this module with respect to the divisibility conditions of the coefficients $b_i$. 

Just as in the case of the module $135\mid 246$ (Section~\ref{ssec:divisibility}), 
one can argue that $\MM=L_I\oplus L_J$
if and only if  $t\mid b_1+b_2$, $t\mid b_3+b_4$, $t\mid b_5+b_6$, $t\mid b_7+b_8$. 
In order to determine the structure of the module $\MM$ when these four divisibility conditions are not fulfilled, 
first  we determine the structure of an endomorphism of this module.  

If $\varphi=(\varphi_i)_{i=0}^7$ is an endomorphism of $\MM$ and $\varphi_0=\begin{pmatrix}a & b \\ c & d  \end{pmatrix}$, then 
\begin{align*}\varphi_{2i+1}&=\begin{pmatrix}a+(b_1+\dots+b_{2i+1})t^{-1}c & tb+(d-a)(b_1+\dots+b_{2i+1})-(b_1+\dots+b_{2i+1})^2t^{-1}c \\ t^{-1}c& d-(b_1+\dots+b_{2i+1})t^{-1}c  \end{pmatrix},\\
\varphi_{2i}&=\begin{pmatrix}a+(b_1+\dots+b_{2i})t^{-1}c & b+t^{-1}((d-a)(b_1+\dots+b_{2i})-(b_1+\dots+b_{2i})^2t^{-1}c) \\ c& d-(b_1+\dots+b_{2i})t^{-1}c  \end{pmatrix}, 
\end{align*}
where $t\mid c$, and 
\begin{align} \label{div} 
t&\mid (d-a)(b_1+b_2)-(b_1+b_2)^2t^{-1}c,\\ \nonumber
t&\mid (d-a)(b_1+b_2+b_3+b_4)-(b_1+b_2+b_3+b_4)^2t^{-1}c,\\   \nonumber
t&\mid (d-a)(b_1+b_2+b_3+b_4+b_5+b_6)-(b_1+b_2+b_3+b_4+b_5+b_6)^2t^{-1}c. \nonumber
\end{align}
 
We distinguish between different cases depending on whether the sums $b_1+b_2$, $b_3+b_4$, $b_5+b_6$, and $b_7+b_8$ are divisible by $t$ or not. We will call these the four {\em divisibility conditions} $t\mid b_1+b_2$, $t\mid b_3+b_4$, 
$t\mid b_5+b_6$ and $t\mid b_7+b_8$, 
and write {\rm (div)} to abbreviate. 
There are three base cases: one of the sums is divisible by $t$ and three are not, 
two are divisible by $t$ and two are not, and none of the sums is divisible by $t$. 
We will see that $\MM$ is indecomposable in the first case and partly in the third case. 
We will explain how the module decomposes 
in the other cases. Furthermore, we will also show that there are infinitely many non-isomorphic modules with the same filtration for the indecomposable case when none of the sums is divisible by $t$.

\subsection{Only one of the sums is divisible by $t$ (Case 1)} 
We first assume that  $t\nmid b_1+b_2$, $t\nmid b_3+b_4$,  $t\nmid b_5+b_6$, and  $t\mid b_7+b_8$. 

\begin{theorem} \label{M1234}
The above defined module $\MM$ is indecomposable if $t\nmid b_1+b_2$, $t\nmid b_3+b_4$,  $t\nmid b_5+b_6$, and  $t\mid b_7+b_8$.  
\end{theorem}

\begin{proof} As in the proof of  Theorem~\ref{t6}  for the module $135\mid 246$, we repeat the same arguments using the divisibility conditions 
\begin{align*}
 t&\mid (d-a)(b_1+b_2)-(b_1+b_2)^2t^{-1}c,\\ 
  t&\mid (d-a)(b_1+b_2+b_3+b_4)-(b_1+b_2+b_3+b_4)^2t^{-1}c,  
\end{align*}   
to conclude that the only possible idempotent endomorphisms of $\MM$ are the trivial ones. We only note that $t\nmid b_1+b_2+b_3+b_4$ because if it were not so, then $t\mid b_5+b_6$ which is not true.  
\end{proof}

In the previous theorem it suffices to choose $b_1=0$, $b_2=1$, $b_3=2$, $b_4=0$, $b_5=0$, $b_6=-3$, $b_7=-1$, and $b_8=1$ in order to fulfil the conditions of the theorem and to have an indecomposable module. 

In the next theorem we show that this module only depends on the divisibility conditions of the coefficients $b_i$, so if we have two different $8$-tuples satisfying the same divisibility conditions, then they give rise to isomorphic modules. 

\begin{theorem}\label{thm:3-cond-iso}
Let  $(c_1, c_2, c_3, c_4, c_5, c_6,c_7,c_8 )$ be  an $8$-tuple  such that $t\nmid c_1+c_2$, $t\nmid c_3+c_4$,  $t\nmid c_5+c_6$, and  $t\mid c_7+c_8$.  If $\MM'$ is the module determined by this $8$-tuple, then the modules $\MM'$ and $\MM$ are isomorphic. 
\end{theorem}

\begin{proof} As in the proof of  Theorem~\ref{tiso} for the module  $135\mid 246$, we explicitly construct an isomorphism $\varphi=(\varphi_i)_{i=0}^7$ between the two modules, where $\varphi_i: V_i\longrightarrow W_i$, and $V_i$ and $W_i$ are the vector spaces at vertex $i$ of the modules $\MM$ and $\MM'$ respectively.  

Let us assume that $\varphi_0=\begin{pmatrix}\alpha & \beta \\ \gamma & \delta  \end{pmatrix}$. Then by repeating the same calculations as for the module $135\mid 246$ we get
\begin{align*}
\varphi_{2i+1}&= \begin{pmatrix} \alpha +(c_1+\dots +c_{2i+1})t^{-1}\gamma & \beta t-\alpha \sum_{j=1}^{2i+1} b_j+\delta \sum_{j=1}^{2i+1} c_j  -(\sum_{j=1}^{2i+1} b_j)(\sum_{j=1}^{2i+1} c_j)t^{-1}\gamma  \\ t^{-1}\gamma  & \delta -(b_1+\dots b_{2i+1})t^{-1}\gamma_0 \end{pmatrix} ,\\
\varphi_{2i}&=\begin{pmatrix} \alpha +(c_1+\dots +c_{2i})t^{-1}\gamma & \beta +t^{-1}(-\alpha \sum_{j=1}^{2i} b_j+\delta \sum_{j=1}^{2i} c_j -t^{-1}\gamma\sum_{j=1}^{2i} b_j \sum_{j=1}^{2i} c_j) \\  \gamma & \delta -(b_1+\dots +b_{2i})t^{-1}\gamma \end{pmatrix}, 
\end{align*}
where $t\mid \gamma$ and 
\begin{align*}
t&\mid -\alpha(b_1+b_2)+(c_1+c_2)\delta -(b_1+b_2)(c_1+c_2)t^{-1}\gamma,\\
t&\mid -\alpha (b_1+b_2+b_3+b_4)+(c_1+c_2+c_3+c_4)\delta  -(b_1+b_2+b_3+b_4)(c_1+c_2+c_3+c_4)t^{-1}\gamma ,\\
t&\mid -\alpha (b_1+b_2+b_3+b_4+b_5+b_6)+(c_1+c_2+c_3+c_4+c_5+c_6)\delta -t^{-1}\gamma \sum_{i=1}^6b_i \sum_{i=1}^6c_i.
\end{align*}

Since $t\mid \gamma$ and we would like $\varphi$ to be invertible, then it must be that $t\nmid \alpha$ and $t\nmid \delta$. Then the inverse of $\varphi_0$ is $\frac{1}{\alpha \delta -\beta \gamma }\begin{pmatrix} \delta &-\beta  \\  -\gamma  & \alpha  \end{pmatrix}.$

In order to find an isomorphism $\varphi$, note that the last divisibility condition is fulfilled because 
$t\mid b_7+b_8$ and $t\mid c_7+c_8$.  Also note that the condition 
$t\mid -\alpha(b_1+b_2+b_3+b_4)+(c_1+c_2+c_3+c_4)\delta -(b_1+b_2+b_3+b_4)(c_1+c_2+c_3+c_4)t^{-1}\gamma$ 
is equivalent to the condition $t\mid\alpha(b_5+b_6)-(c_5+c_6)\delta -(b_5+b_6)(c_5+c_6)t^{-1}\gamma.$  Now we repeat 
the same calculations as for the module $135\mid 246$. 

Note that  there are no conditions attached to $\beta$ so we set it to be 0.  If we set 
\begin{align*}
-\alpha (b_1+b_2)+(c_1+c_2)\delta  -(b_1+b_2)(c_1+c_2)t^{-1}\gamma &=0,\\
\alpha (b_5+b_6)-(c_5+c_6)\delta  -(b_5+b_6)(c_5+c_6)t^{-1}\gamma&=0,
\end{align*}
then we get 
$$\alpha (b_5+b_6)\left[1+\frac{c_5+c_6}{c_1+c_2} \right]-\delta (c_5+c_6)\left[1+\frac{b_5+b_6}{b_1+b_2} \right]=0.$$

If $t\mid 1+\frac{c_5+c_6}{c_1+c_2}$, then from  $\sum_{i=1}^8c_i=0$, we get $t\mid c_3+c_4$, which is not true.  The same holds for $1+\frac{b_5+b_6}{b_1+b_2}$, so both of these elements are invertible. Thus, if we set $\delta=1$, then we get $$\alpha=\frac{(c_1+c_2)(b_3+b_4)(c_5+c_6)}{(b_1+b_2)(c_3+c_4)(b_5+b_6)},$$ and 
$$\gamma=t \frac{(c_3+c_4)(b_5+b_6)-(b_3+b_4)(c_5+c_6)}{(b_1+b_2)(c_3+c_4)(b_5+b_6)}.$$
Hence, 
{\renewcommand\arraystretch{2.5}
$$\varphi_0=\begin{pmatrix} \displaystyle\frac{(c_1+c_2)(b_3+b_4)(c_5+c_6)}{(b_1+b_2)(c_3+c_4)(b_5+b_6)} & \,\,\,\,\,\,\,\, 0 \\  t \displaystyle\frac{(c_3+c_4)(b_5+b_6)-(b_3+b_4)(c_5+c_6)}{(b_1+b_2)(c_3+c_4)(b_5+b_6)} & \,\,\,\,\,\,\,\,1  \end{pmatrix}.$$}
The other invertible matrices $\varphi_i$ are now determined from the above equalities. 
Note that all of them are invertible because their determinant is equal to $\alpha\delta-\beta\gamma$ which is an 
invertible element.  
\end{proof}

{We denote the unique module (up to isomorphism) from Theorem~\ref{M1234} by $\MM_{7,8}$}. It is obvious, due to 
the symmetry of the arguments, that there are also modules $\MM_{1,2}$, $\MM_{3,4}$ and $\MM_{5,6}$ that correspond 
to the remaining three possible divisibility conditions for Case 1, e.g.\ $\MM_{1,2}$ corresponds to the case when 
$t\mid b_1+b_2$, $t\nmid b_3+b_4$,  $t\nmid b_5+b_6$, and  $t\nmid b_7+b_8$. In the next statement we prove that no 
two of these modules are isomorphic to each other. 

\begin{prop} \label{M12}
There are no isomorphic modules  amongst $\MM_{1,2}$, $\MM_{3,4}$, $\MM_{5,6}$ and $\MM_{7,8}$. 
\end{prop}

\begin{proof} Due to the symmetry of the arguments, we only show that $\MM_{7,8}$ is not isomorphic to any of the other 
modules. Assume otherwise, that $\MM_{7,8}$ is isomorphic to $\MM_{i,i+1}$, where $i$ is $1$, $3$ or $5$. Then there is an 
isomorphism between these two modules. Keeping the same notation from the proof of the Theorem~\ref{thm:3-cond-iso}, 
we have that this isomorphism has to satisfy the following divisibility conditions:
\begin{align*}
t&\mid -\alpha (b_1+b_2)+(c_1+c_2)\delta -(b_1+b_2)(c_1+c_2)t^{-1}\gamma,\\
t&\mid -\alpha (b_1+b_2+b_3+b_4)+(c_1+c_2+c_3+c_4)\delta  -(b_1+b_2+b_3+b_4)(c_1+c_2+c_3+c_4)t^{-1}\gamma ,\\
t&\mid -\alpha (b_1+b_2+b_3+b_4+b_5+b_6)+(c_1+c_2+c_3+c_4+c_5+c_6)\delta -t^{-1}\gamma \sum_{i=0}^6 b_i\sum_{i=0}^6 c_i.
\end{align*} Here, the coefficients $b_j$ correspond to $\MM_{7,8}$ and $c_j$ correspond to $\MM_{i,i+1}$. Since $t\mid b_1+b_2+b_3+b_4+b_5+b_6=-(b_7+b_8)$, from the last condition it follows that $t\mid \delta (c_1+c_2+c_3+c_4+c_5+c_6)$. Since $t\nmid \delta$, it must be $t\mid (c_1+c_2+c_3+c_4+c_5+c_6)$. Then $t\mid c_7+c_8=-(c_1+c_2+c_3+c_4+c_5+c_6)$ which is in contradiction with our assumption that $t\nmid c_7+c_8$. 
\end{proof}
  
It follows from the previous proposition that we have now constructed four non-isomorphic rank 2 modules whose filtration is $L_{1357}\mid L_{2468}$. {Before we show that in fact there are infinitely many,  
we consider the other two cases for the divisibility conditions.}

\subsection{Exactly two of the sums are divisible by $t$ (Case 2)}

There are two subcases. The first subcase is when the divisible sums are consecutive, e.g.\ when $t\mid b_1+b_2$, 
$t\mid b_3+b_4$, $t\nmid b_5+b_6$ and $t\nmid b_7+b_8$. The second subcase is when the divisible sums are 
not consecutive, e.g.\ when $t\mid b_1+b_2$, $t\mid b_5+b_6$, $t\nmid b_3+b_4$ and $t\nmid b_7+b_8$. 

Assume first that $t\mid b_1+b_2$, $t\mid b_3+b_4$,  $t\nmid b_5+b_6$, and  $t\nmid b_7+b_8$. 

\begin{theorem}
If $t\mid b_1+b_2$, $t\mid b_3+b_4$,  $t\nmid b_5+b_6$, and  $t\nmid b_7+b_8$, then the module $\MM$ is isomorphic to $L_{\{1,3,5,6\}}\oplus L_{\{2,4,7,8\}}$. 
\end{theorem}
\begin{proof} We show that $\MM$ is decomposable by constructing a non-trivial idempotent endomorphism of $\MM$. Recall that an endomorphism $\varphi=(\varphi_i)_{i=0}^7$ of $\MM$, where  $\varphi_0=\begin{pmatrix}a & b \\ c & d  \end{pmatrix}$, satisfies the divisibility conditions  (\ref{div}).

Since  $t\mid b_1+b_2$, $t\mid b_3+b_4$, these conditions reduce to a single condition $t\mid (d-a)(b_5+b_6)-(b_5+b_6)^2t^{-1}c$. From $t\nmid b_5+b_6$ we conclude that $t\mid (d-a)-(b_5+b_6)t^{-1}c$. 

To construct a non-trivial idempotent homomorphism, as in the case $n=6$, we set $\alpha=1$, $\delta=0=\beta$, and $\gamma=-t(b_5+b_6)^{-1}$. Thus, 
$$\varphi_0=\begin{pmatrix}
1& 0\\
-t(b_5+b_6)^{-1}&0
\end{pmatrix}.$$ 
 Its orthogonal complement is the idempotent
 $$
 \begin{pmatrix}
0& 0\\
t(b_5+b_6)^{-1}&1
\end{pmatrix}.$$ 
Since these are non-trivial idempotents, it follows that the module $\MM$ is decomposable. It remains to show that $\MM\cong L_{\{2,4,7,8\}}\oplus L_{\{1,3,5,6\}}.$  We know that $\MM$ is the direct sum of rank 1 modules $L_X$ and $L_Y$ for some $X$ and $Y$. Let us determine $X$ and $Y$. For this, we take, at vertex $i$,  eigenvectors $v_i$ and $w_i$  corresponding to the eigenvalue $1$ of  the idempotents $\varphi_i$ and $1-\varphi_i$ respectively.  For example, $v_0=[1\,\, ,\,\, -t(b_5+b_6)^{-1}]^t$, $w_0=[0\,\, ,\,\, 1]^t$, $v_1= [ 1-b_1(b_5+b_6)^{-1}  \,\, ,\,\, -(b_5+b_6)^{-1} ]^t$ and  $w_1= [b_1  \,\, , \,\, 1  ]^t$, and so on. A basis for $L_X$ is $\{v_i \mid i=0,\ldots, 7\}$, and a basis for $L_Y$ is $\{w_i \mid i=0,\ldots, 7\}$. Direct computation gives us that $x_1v_0=tv_1$, $x_2v_1=v_2$, $x_3v_2=tv_3$, $x_4v_3=v_4$, $x_5v_4=tv_5$, $x_6v_5=tv_6$,  $x_7v_6=v_7$, $x_8v_7=v_0$. Thus, $X=\{2,4,7,8\}$. Analogously, $Y=\{1,3,5,6\}.$  
 \end{proof}

\begin{rem} In the case when $t\nmid b_1+b_2$, $t\mid b_3+b_4$, $t\mid b_5+b_6$, and $t\nmid b_7+b_8$,  $\MM$ is the direct sum $L_{\{3,5,7,8\}}\oplus L_{\{1,2,4,6\}}$. Similarly, by suitable renaming of the vertices of the quiver, we obtain two more direct sums $L_{\{1,2,5,7\}}\oplus L_{\{3,4,6,8\}}$ and $L_{\{1,3,4,7\}}\oplus L_{\{2,5,6,8\}}$ that have $L_J$ as a submodule and $L_I$ as a quotient module, and there are short exact sequences with $L_I$ and $L_J$ as end terms: 
\begin{align*}
(a) \quad 0\longrightarrow L_J\longrightarrow L_{\{1, 3, 5, 6\}}\oplus L_{\{2,4, 7, 8 \}} \longrightarrow L_I \longrightarrow 0,\\
(b) \quad 0\longrightarrow L_J\longrightarrow  L_{\{3,5,7,8\}} \oplus L_{\{1,2,4,6\}}  \longrightarrow L_I \longrightarrow 0,\\
(c) \quad 0\longrightarrow L_J\longrightarrow    L_{\{1,2,5,7\}}\oplus L_{\{3,4,6,8\}}    \longrightarrow L_I \longrightarrow 0,\\
(d) \quad 0\longrightarrow L_J\longrightarrow      L_{\{1,3,4,7\}}\oplus L_{\{2,5,6,8\}}  \longrightarrow L_I \longrightarrow 0.
\end{align*}

{Here, (a) is the case where $t\mid (b_1+b_2)$ and $t\mid (b_3+b_4)$, (b) the case where $t\mid (b_3+b_4)$ and 
$t\mid (b_5+b_6)$, (c) the case where $t\mid (b_5+b_6)$ and $t\mid (b_7+b_8)$, and (d) the case $t\mid (b_7+b_8)$ 
and $t\mid (b_1+b_2)$.}
The {pairs of profiles} of the four {decomposable} modules that appear in the middle in these short exact 
sequences can be pictured as follows: 

\begin{figure}[H]
\begin{center}
\subfloat[$L_{\{1, 3, 5,6\}}\oplus L_{\{2,4, 7,8\}}$]{\includegraphics[width = 6cm]{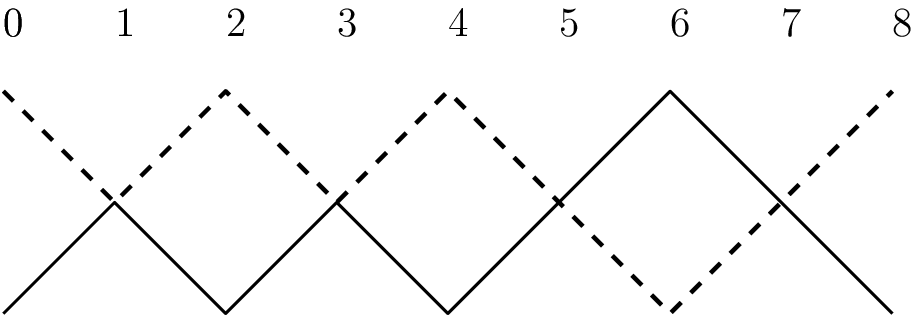}}  \quad \quad\quad\quad 
\subfloat[$L_{\{3,5, 7,8 \}}\oplus L_{\{1, 2, 4, 6\}}$]{\includegraphics[width = 6cm ]{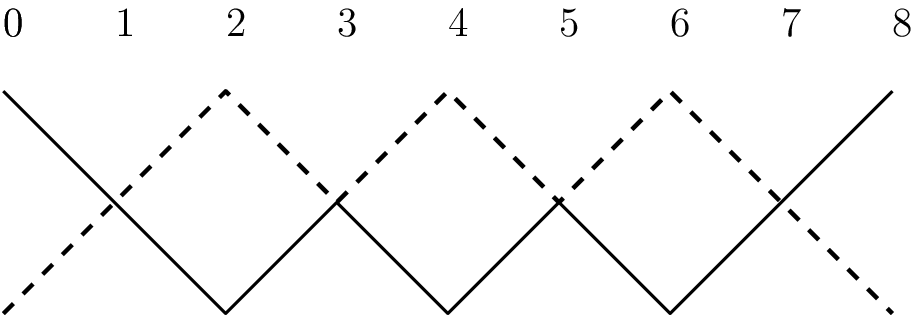}}\\
\subfloat[$L_{\{1, 2, 5 , 7\}}\oplus L_{\{3,4, 6, 8\}}$]{\includegraphics[width = 6cm ]{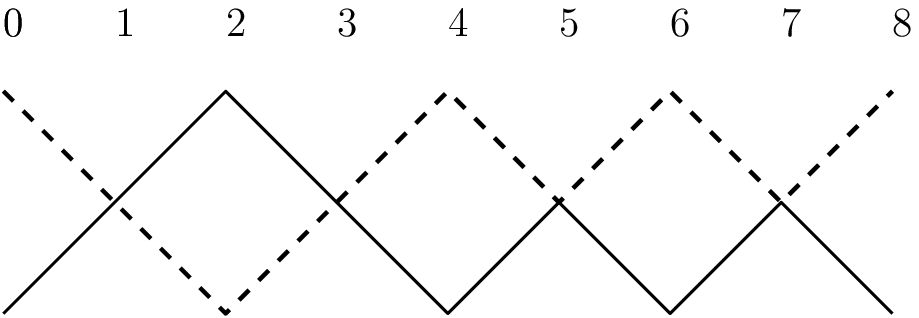}}  \quad \quad \quad \quad
\subfloat[$L_{\{1, 3, 4, 7\}}\oplus L_{\{2,5,6, 8\}}$]{\includegraphics[width = 6cm ]{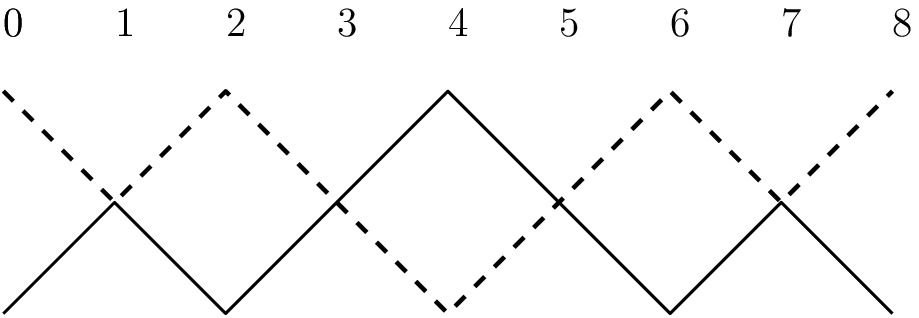}} 
\caption{The {pairs of} profiles of decomposable extensions between $ L_{\{1, 3, 5,7\}}$ and $L_{\{2,4, 6,8\}}$.}
\end{center}
\end{figure}
\end{rem}

Assume now that $t\mid b_1+b_2$, $t\mid b_5+b_6$,  $t\nmid b_3+b_4$, and  $t\nmid b_7+b_8$. 

\begin{theorem}
If $t\mid b_1+b_2$, $t\mid b_5+b_6$,  $t\nmid b_3+b_4$, and  $t\nmid b_7+b_8$, then the module $\MM$ is isomorphic to $L_{\{1,3,4,6\}}\oplus L_{\{2,5,7,8\}}$. 
\end{theorem}
\begin{proof} The only difference from the proof of the previous statement is that the divisibility conditions are now reduced to the condition $t\mid (d-a)-(b_3+b_4)t^{-1}c$. To construct a non-trivial idempotent homomorphism, we set $\alpha=1$, $\delta=0=\beta$, and $\gamma=-t(b_3+b_4)^{-1}$. Thus, 
$$\varphi_0=\begin{pmatrix}
1& 0\\
-t(b_3+b_4)^{-1}&0
\end{pmatrix},$$  
and the rest of proof is analogous to the proof of the previous statement. 
 \end{proof}

\begin{rem} In the case when $t\nmid b_1+b_2$, $t\mid b_3+b_4$, $t\nmid b_5+b_6$, and $t\mid b_7+b_8$,  $\MM$ is the direct sum $L_{\{1,2,4,7\}}\oplus L_{\{3,5,6,8\}}$. 
The {pairs of} profiles of these two modules can be pictured as follows: 
\begin{figure}[H]
\begin{center}
\subfloat[$L_{\{1, 3, 4,6\}}\oplus L_{\{2,5, 7,8\}}$]{\includegraphics[width = 6cm]{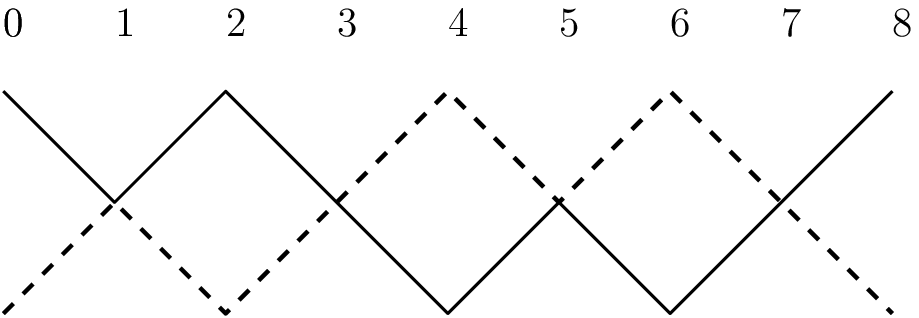}}  \quad \quad\quad\quad 
\subfloat[$L_{\{1, 2, 4, 7\}}\oplus L_{\{3,5, 6,8 \}}$]{\includegraphics[width = 6cm ]{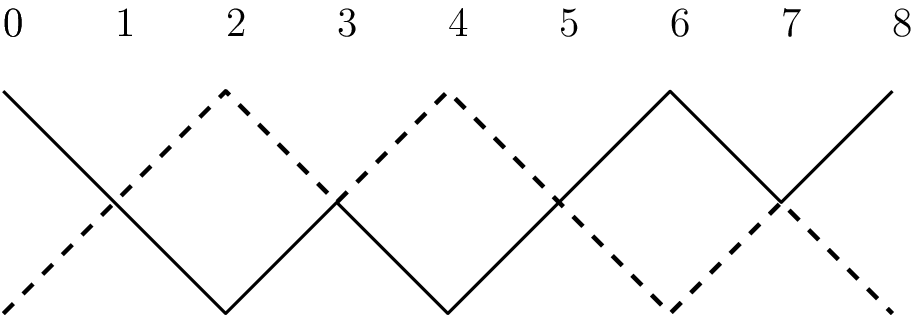}}
\caption{The {pairs of} profiles of decomposable extensions between $ L_{\{1, 3, 5,7\}}$ and $L_{\{2,4, 6,8\}}$.}
\end{center}
\end{figure}
\end{rem}

\subsection{None of the four sums is divisible by $t$ (Case 3)} 

There are {two subcases we have to consider. The first subcase is when all sums $b_i+b_{i+1}+b_{i+2}+b_{i+3}$, 
for $i=1,3,5,7$,  are divisible by $t$, the second subcase is when at least one of these sums is not divisible by $t$. 
In the latter case, we get infinitely many non-isomorphic indecomposable modules as we will show}. 
In this subsection, we always assume that none of the four divisibility conditions $t\mid b_i+b_{i+1}$, 
$i$ odd,  which we continue to abbreviate as {\rm (div)}, is satisfied. 

We first consider the case where all sums $b_i+b_{i+1}+b_{i+2}+b_{i+3}$ are divisible by $t$. 

\begin{theorem}
{Assume that the $(b_i)_i$ satisfy none of the four divisibility conditions {\rm (div)} but that} 
$t\mid b_i+b_{i+1}+b_{i+2}+b_{i+3}$, for $i=1,3,5,7$. Then $\MM\cong L_{\{1,2,5,6\}}\oplus L_{\{3,4,7,8\}} .$
\end{theorem}
\begin{proof}
As before, we will construct a non-trivial idempotent endomorphism of $\MM$ to prove that it is a decomposable module. Also, as before, our endomorphism has to satisfy the conditions (\ref{div}). 

From $t\mid b_1+b_{2}+b_{3}+b_{4}$ it follows that the conditions (\ref{div}) reduce to a single condition 
$t\mid (d-a)-(b_1+b_2)t^{-1}c.$ Note that this condition is equivalent to the condition $t\mid (d-a)-(b_5+b_6)t^{-1}c.$ As before, we obtain a non-trivial idempotent 
$$\varphi_0=\begin{pmatrix}
1& 0\\
-t(b_1+b_2)^{-1}&0
\end{pmatrix}.$$ 
The rest of the proof is analogous to the proof of the previous two statements. 
\end{proof}

\begin{rem} 
The {pairs of} profiles of $L_{\{1,2,5,6\}}\oplus L_{\{3,4,7,8\}}$ can be pictured as follows: 
\begin{figure}[H]
\begin{center}
{\includegraphics[width = 6cm]{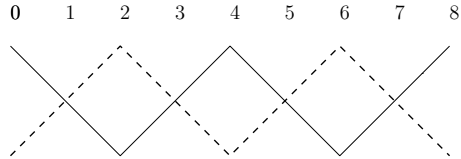}} 
\caption{The {pairs of} profile of a decomposable extension between $ L_{\{1, 3, 5,7\}}$ and $L_{\{2,4, 6,8\}}$.}
\end{center}
\end{figure}
\end{rem}

Assume now that  {for the tuple $(b_i)_i$ one of the consecutive sums of four entries is not divisible by $t$}. 
\begin{prop}\label{prop:more-indec}
{If none of the four divisibility conditions {\rm (div)} holds for $(b_i)_i$ and if there exists an $i\in \{1,3,5,7\}$ 
such that} 
$t\nmid b_i+b_{i+1}+b_{i+2}+b_{i+3}$, 
then $\MM$ is indecomposable.
\end{prop}
\begin{proof}
Assume, without loss of generality, that $t\nmid b_1+b_2+b_3+b_4$. As in the proof of indecomposability of the module $135\mid 246$ the endomorphism of $\MM$ has to satisfy the conditions 
\begin{align*}
t&\mid (d-a)(b_1+b_2)-(b_1+b_2)^2t^{-1}c,\\
t&\mid (d-a)(b_1+b_2+b_3+b_4)-(b_1+b_2+b_3+b_4)^2t^{-1}c,\\  
t&\mid (d-a)(b_1+b_2+b_3+b_4+b_5+b_6)-(b_1+b_2+b_3+b_4+b_5+b_6)^2t^{-1}c.
\end{align*} 
Now we repeat the same arguments as for the module $135\mid 246$. From the first two conditions we get that $t\mid (b_3+b_4)t^{-1}c$. Since $t\nmid b_3+b_4$, it follows that $t\mid t^{-1}c$ and that $t\mid d-a$. If $a+d=1$, then, because $t\mid a$ or $t\mid 1-a$, it must be that $t\mid 1$, which is not true. Thus, it must be $a=d$ and $c=b=0$, giving us only trivial idempotents.
\end{proof}

Now that we know that whenever $t\nmid b_i+b_{i+1}+b_{i+2}+b_{i+3}$, for at least one $i$, the module $\MM$ 
with the given $(b_i)_i$ is indecomposable, we would like to know if the constructed modules are isomorphic. 
{Since $\sum b_i=0$, we have $t\nmid b_i+b_{i+1}+b_{i+2}+b_{i+3}$ for some $i$ if and only if 
$t\nmid b_{i+4}+b_{i+5} + b_{i_6} + b_{i_7}$. Therefore, these conditions come in pairs of ``complementary'' sums.} 
Thus we have to distinguish between the cases when two of the sums $b_i+b_{i+1}+b_{i+2}+b_{i+3}$ are divisible by $t$ 
and two are not, and when none of these sums of four consecutive $b_i$'s is divisible by $t$. {We will see that in the first 
case, we only get two indecomposable modules (up to isomorphism) while in the latter case, we get infinitely many.} 

Let us assume that $(c_i)_1^8$ is another $8$-tuple satisfying none of the divisibility conditions {\rm (div)} and such 
that $t\nmid c_i+c_{i+1}+c_{i+2}+c_{i+3}$, for some $i\in \{1,3,5,7\}$. 
Denote the module given by these $(c_i)_i$ by $\MM'$. 

In the following two propositions we consider the case when both $b_i$'s and $c_i$'s have two of the above mentioned sums divisible by $t$, and two sums not divisible by $t$.  
\begin{prop} \label{prop:two-sums-not-matching}
If $t\nmid b_1+b_2+b_3+b_4$, $t\mid b_3+b_4+b_5+b_6$, $t\nmid c_1+c_2+c_3+c_4$, and $t\mid c_3+c_4+c_5+c_6$, 
then $\MM$ and $\MM'$ are isomorphic.
\end{prop}

\begin{proof} Keeping the same notation as before when constructing isomorphisms, it must hold: 
\begin{align*}
t&\mid -\alpha (b_1+b_2)+(c_1+c_2)\delta -(b_1+b_2)(c_1+c_2)t^{-1}\gamma,\\
t&\mid -\alpha (b_1+b_2+b_3+b_4)+(c_1+c_2+c_3+c_4)\delta  -(b_1+b_2+b_3+b_4)(c_1+c_2+c_3+c_4)t^{-1}\gamma ,\\
t&\mid -\alpha (b_1+b_2+b_3+b_4+b_5+b_6)+(c_1+c_2+c_3+c_4+c_5+c_6)\delta -t^{-1}\gamma \sum_{i=1}^6b_i\sum_{i=1}^6c_i.
\end{align*}
Since $t\mid b_3+b_4+b_5+b_6$ and $t\mid c_3+c_4+c_5+c_6$, the above conditions reduce to the first two conditions. Now, we set  $-\alpha (b_1+b_2)+(c_1+c_2)\delta -(b_1+b_2)(c_1+c_2)t^{-1}\gamma=0$ and $-\alpha (b_1+b_2+b_3+b_4)+(c_1+c_2+c_3+c_4)\delta  -(b_1+b_2+b_3+b_4)(c_1+c_2+c_3+c_4)t^{-1}\gamma=0.$ Subsequently, $\alpha (b_1+b_2+b_3+b_4)(c_3+c_4)(c_1+c_2)^{-1}-\delta (c_1+c_2+c_3+c_4) (b_3+b_4)(b_1+b_2)^{-1} =0$. By setting $\alpha=1$, $\beta=0$, we get that $\delta=[(b_1+b_2)(b_1+b_2+b_3+b_4)(c_3+c_4)]          [(c_1+c_2)(c_1+c_2+c_3+c_4) (b_3+b_4)]^{-1}$, and $\gamma=-(c_1+c_2)^{-1}+\delta (b_1+b_2)^{-1}$, giving us an isomorphism between $\MM$ and $\MM'$.
\end{proof}

\begin{prop}\label{prop:2-long-div-cond}
If $t\nmid b_1+b_2+b_3+b_4$, $t\mid b_3+b_4+b_5+b_6$, $t\mid c_1+c_2+c_3+c_4$, and $t\nmid c_3+c_4+c_5+c_6$, then $\MM$ and $\MM'$ are not isomorphic.
\end{prop} 
\begin{proof}
If there were an isomorphism between $\MM$ and $\MM'$, then its coefficients would satisfy
\begin{align*}
t&\mid -\alpha (b_1+b_2)+(c_1+c_2)\delta -(b_1+b_2)(c_1+c_2)t^{-1}\gamma,\\
t&\mid -\alpha (b_1+b_2+b_3+b_4)+(c_1+c_2+c_3+c_4)\delta  -(b_1+b_2+b_3+b_4)(c_1+c_2+c_3+c_4)t^{-1}\gamma ,\\
t&\mid -\alpha (b_1+b_2+b_3+b_4+b_5+b_6)+(c_1+c_2+c_3+c_4+c_5+c_6)\delta -t^{-1}\gamma \sum_{i=1}^6b_i\sum_{i=1}^6c_i.
\end{align*}
From the second condition we obtain $t\mid \alpha (b_1+b_2+b_3+b_4)$. But $t\nmid b_1+b_2+b_3+b_4$ and $t\nmid \alpha$, which is a contradiction. 
\end{proof}

\begin{rem} \label{rem:noniso}
The same arguments used in the proof of the previous proposition tell us that the two modules $\MM$ and 
$\MM'$ from Proposition~\ref{prop:2-long-div-cond} are two new non-isomorphic indecomposable modules which are not 
isomorphic to any of the modules $\MM_{1,2}$, $\MM_{3,4}$, $\MM_{5,6}$ and $\MM_{7,8}$ constructed before. 
For example, if the $b_i$'s correspond to the module $\MM_{7,8}$ and the $c_i$'s correspond to the module $\MM$, 
then from the third relation {in the proof of Proposition~\ref{prop:two-sums-not-matching}} 
we obtain that  $t\mid \delta (c_1+c_2+c_3+c_4+c_5+c_6)$, yielding $t\mid \delta (c_1+c_2)$, which is not true since 
$t\nmid \delta$ and $t\nmid (c_1+c_2)$.
\end{rem}

We are only left to examine if, in the case when none of the sums $b_i+b_{i+1}+b_{i+2}+b_{i+3}$ is divisible by $t$, 
for two different tuples we obtain isomorphic modules. In the following theorem we assume that $b_i$'s correspond to 
the module $\MM$ and $c_i$'s correspond to the module $\MM'$. Also, we assume that $t\nmid b_i+b_{i+1}$ 
$t\nmid c_i+c_{i+1}$ for odd $i$.

\begin{theorem}  \label{thm:infinite}
If $t\nmid b_i+b_{i+1}+b_{i+2}+b_{i+3}$ and $t\nmid c_i+c_{i+1}+c_{i+2}+c_{i+3}$, for $i=1,3,5,7,$ then the modules $\MM$ 
and $\MM'$ are isomorphic if and only if 
$$
t\mid (b_1+ b_2) (c_3+ c_4) (b_5+ b_6) (c_7+ c_8) -(c_1+ c_2) (b_3+ b_4) (c_5+ c_6) (b_7+ b_8).
$$
\end{theorem}

\begin{proof}
As before, if there were an isomorphism between $\MM$ and $\MM'$, its coefficients would have to satisfy the 
following conditions: 
\begin{align*}
t&\mid -\alpha (b_1+b_2)+(c_1+c_2)\delta -(b_1+b_2)(c_1+c_2)t^{-1}\gamma,\\
t&\mid -\alpha (b_1+b_2+b_3+b_4)+(c_1+c_2+c_3+c_4)\delta  -(b_1+b_2+b_3+b_4)(c_1+c_2+c_3+c_4)t^{-1}\gamma ,\\
t&\mid \alpha (b_7+b_8)-(c_7+c_8)\delta -(b_7+b_8)(c_7+c_8)t^{-1}\gamma.
\end{align*}
From these we get that 
\begin{align*}
t&\mid \alpha (c_3+ c_4)[(c_1+ c_2)(c_1+ c_2+c_3+ c_4)]^{-1}-\delta (b_3+ b_4)[(b_1+ b_2)(b_1+ b_2+b_3+ b_4)]^{-1},\\
t&\mid \alpha (c_5+ c_6)[(c_7+ c_8)(c_1+ c_2+c_3+ c_4)]^{-1}-\delta  (b_5+ b_6)[(b_7+ b_8)(b_1+ b_2+b_3+ b_4)]^{-1}.
\end{align*}
Finally, from the last two relations we get 
$$t\mid \alpha [(b_1+ b_2) (c_3+ c_4) (b_5+ b_6) (c_7+ c_8) -(c_1+ c_2) (b_3+ b_4) (c_5+ c_6) (b_7+ b_8)].$$
If $t\nmid (b_1+ b_2) (c_3+ c_4) (b_5+ b_6) (c_7+ c_8) -(c_1+ c_2) (b_3+ b_4) (c_5+ c_6) (b_7+ b_8)$, then there is no isomorphism between $\MM'$ and $\MM$. If $t\mid (b_1+ b_2) (c_3+ c_4) (b_5+ b_6) (c_7+ c_8) -(c_1+ c_2) (b_3+ b_4) (c_5+ c_6) (b_7+ b_8), $ then we simply set $\alpha=1$, and compute $\delta$ and $\gamma$ from the above relations (as before, we set $\beta=0$).
\end{proof}

\begin{rem}
It is easily shown that none of the indecomposable modules from Theorem~\ref{thm:infinite} is isomorphic to any of 
the indecomposable modules from the previous cases. To prove this, we use the same arguments as in 
Remark~\ref{rem:noniso}.
\end{rem}

We will now parametrize the non-isomorphic indecomposable modules from Theorem~\ref{thm:infinite}.

Let $\beta \in \mathbb C\setminus {\{0,1,-1\}}$.  Keeping the notation from the theorem, choose the parameters 
$b_i$ in the following way: 
$b_1+b_2=-(b_5+b_6)=1$, $b_3+b_4=-(b_7+b_8)=\beta$.  Then, 
$t\nmid b_i+b_{i+1}$ and $t\nmid b_i+b_{i+1}+b_{i+2}+b_{i+3}$, for odd $i$. 
Denote the indecomposable module that corresponds to these coefficients by $\MM_{\beta}$. 

\begin{corollary} \label{cor:inf} There are infinitely many non-isomorphic rank $2$ indecomposable  modules in ${\rm CM}(B_{4,8})$ with profile $1357\mid 2468$.
\end{corollary}
\begin{proof} Let $\alpha\in \mathbb C\setminus {\{0,1,-1\}}$, with $\beta\neq \pm\alpha$, and $\MM_{\alpha}$ be the corresponding indecomposable module. Then $\MM_{\alpha}$ and $\MM_{\beta}$ are not isomorphic. Indeed, assuming that the coefficients $c_i$ correspond to $\MM_{\alpha}$, we have that 
 $$t\nmid (b_1+ b_2) (c_3+ c_4) (b_5+ b_6) (c_7+ c_8) -(c_1+ c_2) (b_3+ b_4) (c_5+ c_6) (b_7+ b_8)=\alpha^{2}-\beta^{2},$$ since $\alpha\neq \pm \beta$, so by the previous theorem the corresponding modules are not isomorphic.
\end{proof}

From the proof of the previous corollary it follows that the two modules $\MM_{\alpha}$ and $\MM_{\beta}$ are isomorphic 
if and only if $\alpha=\pm \beta$. Thus, the non-isomorphic indecomposable modules of this form are parameterized 
by $\mathbb C\setminus {\{0,1,-1\}}$, where two points in this set are identified if they sum up to $0$.
In the next proposition we show that  every indecomposable module as in Theorem~\ref{thm:infinite} 
is isomorphic to $\MM_{\beta}$ for some $\beta$.

\begin{prop}  Let $\MM$ be a rank $2$ indecomposable module with the corresponding coefficients $c_i$ satisfying 
$t\nmid c_i+c_{i+1}$, $t\nmid c_i+c_{i+1}+c_{i+2}+c_{i+3}$, for odd $i$. 
Then there exists $\beta \in \mathbb C\setminus {\{0,1,-1\}}$ such that $\MM$ is isomorphic to $\MM_{\beta}$. 
\end{prop}

\begin{proof} Let $c_i+c_{i+1}=C_i$, for $i=1,3,5,7.$ Since the coefficients of $\MM_{\beta}$ satisfy $b_1+b_2=-(b_5+b_6)=1$, 
$b_3+b_4=-(b_7+b_8)=\beta$, it follows from Theorem~\ref{thm:infinite} that we need to find  $\beta$ 
satisfying $t\mid \beta^2C_1C_5-C_3C_7$. If $\gamma_i$ is the  constant term of $C_i$, then we choose 
$\beta$ to be a square root of $(\gamma_3\gamma_7)(\gamma_1\gamma_5)^{-1}$. 
{Note that by the divisibility conditions {\rm (div)}, 
$\gamma_1\gamma_5\ne 0$ and $\gamma_3\gamma_7\ne 0$. In particular}, $\beta\neq 0$. If $\beta=\pm1$, then 
$\gamma_1\gamma_5=\gamma_3\gamma_7$. From $C_1+C_3+C_5+C_7=0$, we get after multiplying by $C_5$ that 
$(\gamma_3+\gamma_5)(\gamma_5+\gamma_7)=0$ which is not possible. Hence, $\beta\neq \pm1.$
\end{proof}

\section{The general case: $r$-interlacing, $r\geq 4$}\label{sec:general}

In this section we generalize the results from the previous two sections. 
We deal with the case of $r$-interlacing, where $r\geq 4$ and at least four boxes (i.e.\ $r$-interlacing rims where 
$r\ge r_1>3$). 
We prove that if $I$ and $J$ are $r$-interlacing with poset $1^{r_1}\mid 2$, where $r\geq 4$ and $r\geq r_1> 3$, 
then there exist non-isomorphic indecomposable modules with the given filtration $L_I\mid L_J$. It follows that in this case, 
as in Section~\ref{sec:4-interlacing}, 
the profile of a module does not uniquely determine the module.  
Let $I$ and $J$ be $r$-interlacing and form $r_1$ boxes. In general, $I\cap J$ and $I^c\cap J^c$ are non-empty. We have to modify our definition of branching points $i_m$ and associated points $j_m$ for the general setting.

\begin{defn}\label{def:branching} 
Let $I$ and $J$ be two $k$-subsets such that their lattice diagram forms $r_1$ boxes. 
The {\em branching points of the lattice diagram 
$I \mid J$} are defined to be the points where the boxes end, i.e.\ $i\in I\setminus J$ is a branching point if  $i+1\notin I$  and  the two rims meet at $i$. We denote them by $\{i_1,i_2,\dots, i_{r_1}\}$. In addition, we define the points  
$\{j_1,j_2,\dots,j_{r_1}\}$ {at the beginning of the boxes} 
as the set of $j\in J\setminus I$ such that $j-1\notin J$ and such that $j_m$ is minimal in 
$\{i_m+1,\dots, i_{m+1}\}$ (cyclically) with this property. 
The {\em size} of the box ending at $i_m$ is defined to be the number 
of elements of $I\setminus J$ for that box. 
\end{defn}

Let $r\geq 4$. If $k$ is arbitrary and $I$, $J$ are such that $I$ and $J$ are $r$-interlacing and  $I\mid J$ has poset 
of the form $1^{r_1}\mid 2$, where $r\geq r_1>3$ (see Figure~\ref{fig9}), then we are able to construct more than one  non-isomorphic 
indecomposable rank 2 module which has $L_J$ as a submodule and $L_I$ as the quotient as follows. 
Denote this module by $\MM(I,J)$. We mimic the same procedure as for the module $1357\mid 2468$.   

\begin{figure}[H] 
\begin{center}
\includegraphics[width = 8.3cm]{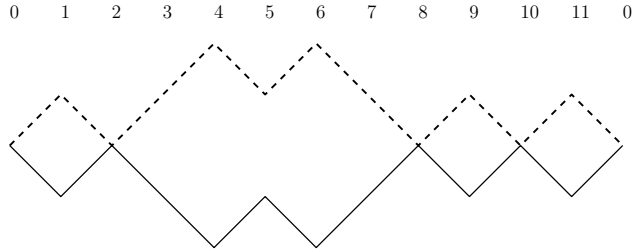}
\caption{The profile of  a  module  with $5$-interlacing layers and  with poset $1^4\mid 2$, for $(k,n)=(6,12)$.}\label{fig9}
\end{center}
\end{figure}

{The two rims form $r_1$ boxes. Denote the $r_1$ branching points from $I$ where the boxes of the two rims end by 
$i_1,\dots, i_{r_1}$ 
and their counterparts in $J$ by $j_1,\dots, j_{r_1}$, as in Definition~\ref{def:branching}.} 

For these branching points, we set 
$x_{i_l}=\begin{pmatrix} t& b_{i_l} \\ 0 & 1 \end{pmatrix}$ and $x_{j_l}=\begin{pmatrix} 1& b_{j_l} \\ 0 & t \end{pmatrix}$ for 
$l=1,\dots, r_1$. For all other vertices $i$ we define $x_i$ (resp.\ $y_i$) to be diagonal matrices as follows:  the diagonal 
of $x_i$ 
(resp.\ $y_i$) is $(1,t)$ (resp.\ $(t,1)$) if $i\in J\setminus I$, it is $(t,1)$ (resp.\ $(1,t)$)  if $i\in I\setminus J$, $(t,t)$ 
(resp.\ $(1,1)$) if $i\in I^c\cap J^c$, and $(1,1)$ (resp.\ $(t,t)$) if $i\in I\cap J$. 
We also assume that 
$\sum_1^{r_1}(b_{i_l}+b_{j_l})=0$ so that we have a module structure.

Now we assume that the following divisibility conditions hold for the $b_i$'s at the first three branching points, 
as in Theorem~\ref{M1234} (recalling from its proof that $t\nmid b_{1}+b_{2}+b_{3}+b_{4}$) 
for the module $1357\mid 2468$:  $t\nmid b_{i_1}+b_{j_1}$, $t\nmid b_{i_2}+b_{j_2}$, $t\nmid b_{i_3}+b_{j_3}$, 
with 
$t\nmid b_{i_1}+b_{j_1}+b_{i_2}+b_{j_2}$, 
and $t\mid b_{i_l}+b_{j_l}$ for $l\geq 4$. As in the previous sections, it is now easy to prove that this module is 
indecomposable by invoking the same divisibility arguments as before. 
Obviously, we could start at any branching point in order to obtain additional {$r_1-1$} indecomposable modules. As in  Proposition~\ref{M12}, one can easily prove that no two of these $r_1$ indecomposable rank 2 modules are isomorphic.  
  
Therefore, we have the following proposition.

\begin{prop} 
If $I$ and $J$ are $r$-interlacing and $I\mid J$ has the poset $1^{r_1}\mid 2$, where $r\geq r_1> 3$, then there  are
more than one indecomposable rank $2$ modules with the profile $I\mid J$.
\end{prop}
 
Furthermore, it is easy to adapt the proof of Theorem \ref{thm:infinite} and Corollary \ref{cor:inf} to the general case when 
$r\geq 4$ in order to obtain the following theorem (we omit the proof).
 
\begin{theorem} \label{thm:infmod}
Let $I$ and $J$ be $r$-interlacing with poset $1^{r_1}\mid 2$, where $r\geq r_1> 3$. There are infinitely many non-isomorphic 
rank $2$ indecomposable modules in  ${\rm CM}(B_{k,n})$ with profile $I \mid J$.
\end{theorem}

For  given  $r$-interlacing $k$-subsets $I$ and $J$ with  poset $1^{r_1}\mid 2$, where $r\geq r_1> 3$, we note that as 
the number $r_1$ increases the parameterization of non-isomorphic indecomposable rank 2 modules with filtration 
$L_I\mid L_J$ becomes more complicated. In the case $r=r_1=4$, we have seen that the family of non-isomorphic 
indecomposable modules with filtration $L_I\mid L_J$ is parameterized by the set $\mathbb C\setminus {\{0,1,-1\}}$ 
up to sign {(if $\alpha=-\beta$, then $\MM_{\alpha}\cong \MM_{\beta}$)}. 
Here, we do not pursue the classification of these 
non-isomorphic indecomposable modules, but it would be nice to have this sort of classification for general $r\geq r_1$.


\bibliographystyle{abbrv}
\bibliography{biblio}

\end{document}